\documentclass[11pt,preprint]{imsart}
\RequirePackage[OT1]{fontenc}
\RequirePackage{amsthm,amsmath}
\usepackage{stmaryrd}
\usepackage{graphicx,amsmath,amsthm,amsfonts,amssymb}
\usepackage{booktabs, rotating, float,pifont}
\usepackage[hmargin=3cm,vmargin=2.4cm]{geometry}
\usepackage{mathtools}
\usepackage{bbm}
\mathtoolsset{showonlyrefs}
\usepackage{enumerate}
\usepackage{enumitem}
\usepackage{color}
\usepackage{listings}
\usepackage{epstopdf}
\usepackage{epsfig}
\usepackage{subfigure}
\usepackage{makecell}
\usepackage{tikz}
\usetikzlibrary{decorations.pathreplacing}
\usepackage{array}
\usepackage{caption}
\usepackage{soul}
\usepackage{hyperref}
\usepackage{verbatim}
\usepackage{siunitx}
\usetikzlibrary{calc} \tikzset{>=latex}
\usetikzlibrary{calc,decorations.markings}
\usetikzlibrary{positioning,shapes,decorations.text}
\theoremstyle{plain}
\newtheorem{thm}{Theorem}[section]
\newtheorem{cor}[thm]{Corollary}

\newtheorem{prop}[thm]{Proposition}
\newtheorem{lemma}[thm]{Lemma}
\newtheorem{assumption}[thm]{Assumption}
\theoremstyle{definition}
\newtheorem{defn}[thm]{Definition}
\newtheorem{rem}[thm]{Remark}
\newtheorem{eg}[thm]{Example}
\newtheorem{fact}[thm]{Fact}
\newtheorem{observe}[thm]{Observation}

\numberwithin{equation}{section}

\newcommand{\rpm}{\sbox0{$1$}\sbox2{$\scriptstyle\pm$}
  \raise\dimexpr(\ht0-\ht2)/2\relax\box2 }

\tikzstyle{nd} = [anchor=base, inner sep=0pt]
\tikzstyle{ndpic} = [remember picture, baseline, every node/.style={nd}]

\def\beq{\begin{equation}}
\def\eeq{\end{equation}}
\def\ba{\begin{enumerate}[(a)]}
\def\bei{\begin{enumerate}[(i)]}
\def\be{\begin{enumerate}[(1)]}
\def\ee{\end{enumerate}}
\def\bi{\begin{itemize}}
\def\ei{\end{itemize}}
\def\beg{\begin{eg}}
\def\eeg{\end{eg}}
\def\bd{\begin{defn}}
\def\ed{\end{defn}}
\def\bt{\begin{thm}}
\def\et{\end{thm}}
\def\bl{\begin{lemma}}
\def\el{\end{lemma}}
\def\bfac{\begin{fact}}
\def\efac{\end{fact}}

\def\bc{\begin{cor}}
\def\ec{\end{cor}}
\def\bp{\begin{prop}}
\def\ep{\end{prop}}
\def\bo{\begin{observe}}
\def\eo{\end{observe}}
\def\bas{\begin{assumption}}
\def\eas{\end{assumption}}

\def\ZZ{\mathbb{Z}}

\def\beg{\begin{eg}}
\def\eeg{\end{eg}}

\def\Tau{\mathcal{T}}

\def\type{{\textrm{type}}}

\def\ua{{\Uparrow}}

\def\sh{{\textrm{sh}}}
\newcommand{\cross}{\begin{picture}(12,10)(-2,0)
\put(-8,0){\Large$\longrightarrow$}
\put(0,0){\Large$\big\downarrow$}
\end{picture}}

\def\ddd{\displaystyle}

\makeindex

\numberwithin{equation}{section}
\numberwithin{table}{section}

\startlocaldefs
\endlocaldefs

\begin{document}

\begin{frontmatter}

\title{Invariance of polymer partition functions under the geometric RSK correspondence}
\runtitle{Invariance of polymer partition functions under the geometric RSK correspondence}
\runauthor{Corwin}

\begin{aug}
 \author{\fnms{Ivan}  \snm{Corwin}\corref{}\ead[label=e1]{corwin@math.columbia.edu}}
%
\address{Department of Mathematics, 2990 Broadway, New York, NY 10027, USA\\ \printead{e1}}




\end{aug}

\begin{abstract}
 We prove that the values of discrete directed polymer partition functions involving multiple non-intersecting paths remain invariant under replacing the background weights by their images under the geometric RSK correspondence. This result is inspired by a recent and remarkable identity proved by Dauvergne, Orthmann and Vir\'ag which is recovered as the zero-temperature, semi-discrete limit of our main result.
\end{abstract}




\end{frontmatter}


\section{Introduction}

The {\it Robinson-Schensted-Knuth (RSK) correspondence} \cite{Rob,Schen,Knuth} plays an central role in combinatorics and  symmetric function theory \cite{Stanley,Sagan}. In the past twenty years, starting with  work of \cite{Kurt00,BDJ}, it has also taken on a key role in the study of certain models in integrable probability such as $(1+1)$-dimensional {\it last passage percolation (LPP)}. Greene's theorem \cite{Greene} provides the link between the RSK correspondence and LPP. The push-forward under the RSK correspondence of an array of independent and identically distributed geometric random variables produces a measure on pairs of semi-standard Young tableaux which can be described in terms of Schur symmetric polynomials.

This connection between geometric weight LPP and the {\it Schur process/measure} \cite{Okounkov2001,OkRes} combined with the {\it determinantal} structure behind these measures serves as a starting point for computing asymptotics. Greene's theorem implies that the last passage times from the origin to various points along down-right (i.e., {\it space-like}) paths can be read-off from directly from the output of the RSK correspondence. In the geometric weight setting, the asymptotic joint distribution of these last passage times are described by the {\it Airy$_2$ process} \cite{Prahofer2002,JohPNG}. The Airy$_2$ process is the top curve of the Airy line ensemble \cite{CH14} which records the limit of the entire Schur process. The lower curves of the Airy line ensemble relate to limits of last passage times for multipaths.

These asymptotics fit into the study of the {\it Kardar-Parisi-Zhang (KPZ) universality class} \cite{Corwin12,Quastel12,QS15}. Based on the aforementioned LPP results, the Airy$_2$ process has been understood to be the universal process which governs the fixed starting point and spatially varying ending point distribution of the KPZ universality class. The one-parameter Airy$_2$ process was conjectured in \cite{Corwin2015} to be a marginal of similarly conjectural two-parameter process termed the {\it Airy sheet} which should describe the limiting joint law of LPP (and other KPZ class model) fluctuations under spatially varying both the starting and ending points. Given the Airy sheet, one can construct the full KPZ fixed point (i.e., the universal space-time limit of models in the KPZ universality class) -- see also \cite{KPZfixed}.

\cite[Theorem 1.3]{DOV} constructs the {\it Airy sheet} from the {\it Airy line ensemble}. Instead of studying geometric weight LPP, \cite{DOV} focuses on a semi-discrete limit known as Brownian LPP. The starting point for the work of \cite{DOV} is a remarkable generalization of Greene's theorem in the semi-discrete setting which shows how the collection of last passage times from spatially varying starting and endpoint points is encoded simply from the (semi-discrete) RSK correspondence output. Their result, \cite[Proposition 4.1]{DOV}, is Corollary \ref{cor:sdzero} in this text.

\smallskip
Our present paper addresses the question of whether \cite[Proposition 4.1]{DOV} generalize to the {\it discrete} RSK correspondence and to the {\it geometric} RSK correspondence?
Theorem \ref{thm:dinv} answers both of these questions in the affirmative by providing a generalization of \cite[Proposition 4.1]{DOV} for the discrete geometric RSK correspondence.

\smallskip
The {\it geometric RSK (gRSK) correspondence} \cite{Kir,NY, ON12, COSZ,OSZ} (also sometimes called the {\it tropical RSK correspondence}) is the image of the usual RSK correspondence under replacement of the $(\max,+)$ semi-ring by the $(+,\times)$ semi-ring. It is important to note that the word {\it geometric} has been used in two ways so far in this introduction -- in reference to geometric random variables in LPP and in reference to this geometric lifting of the RSK correspondence. We will generally use {\it geometric} in the second sense in this paper.

Greene's theorem generalizes to the gRSK correspondence and provides a relationship to directed polymer partition functions (instead of LPP as for usual RSK). The special geometric random variables which, under RSK, produced the Schur process, have an analog for gRSK too -- the image of inverse-gamma random variables under the gRSK correspondence maps to the {\it Whittaker process/measure} \cite{ON12,COSZ,OSZ}.

Our main result is a generalization of \cite[Proposition 4.1]{DOV} to the discrete geometric RSK correspondence. We briefly explain this here, leaving precise definitions to the main text.

\smallskip
Consider  any $n\in \mathbb{Z}_{\geq 2}$, $k\in \mathbb{Z}_{\geq 1}$, $U=\big((u_1,n),\ldots, (u_k,n)\big)$ and $V=\big((v_1,1),\ldots, (v_k,1)\big)$ with $u_1< \cdots < u_k\in \mathbb{Z}_{\geq 1}$ and $v_1<\cdots<x v_k\in \mathbb{Z}_{\geq 1}$. Define $U\to V$ to be the set of  multipaths $\pi=(\pi_1,\ldots,\pi_k)$ where each $\pi_i$ is a directed lattice path starting at $(u_i,n)$ and ending at $(v_i,1)$, and the $\pi_i$ are pairwise non-intersecting. Assume that $U\to V$ is non-empty. The left-side of Figure \ref{fig:Uhat} illustrates a choice of $U$, $V$ and $\pi\in U\to V$ when $n=5$ and $k=3$.

\begin{figure}[t]
  \captionsetup{width=.8\linewidth}
  \includegraphics[width=6in]{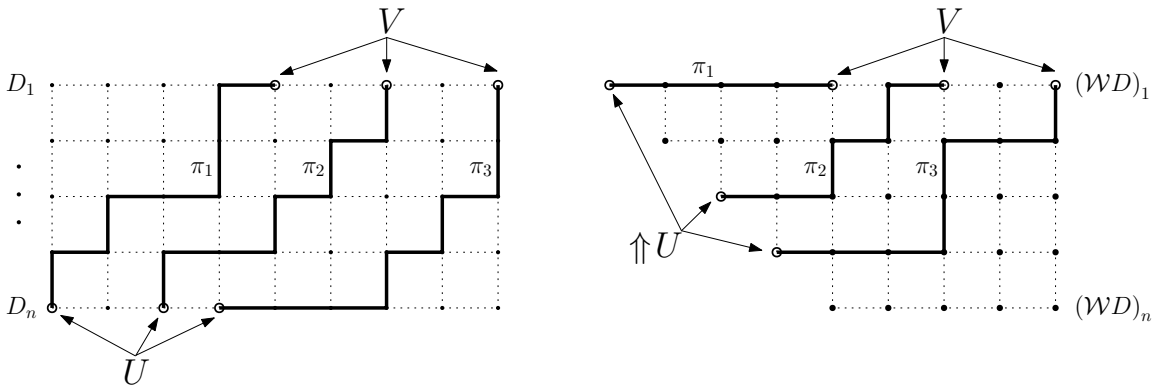}
  \caption{An illustration of many of the terms used to formulate our main result, Theorem \ref{thm:dinv} when $n=5$ and $k=3$. The mapping from $U$ to $\ua U$ as defined in Theorem \ref{thm:dinv}. The circles correspond to the $U$ and $V$ on the left or $\ua U$ and $V$ on the right, and $\pi_1,\pi_2$ and $\pi_3$ represent possible non-intersecting paths which go from $U$ to $V$ on the left, or $\ua U$ to $V$ on the right.  Also illustrated is a portion (left-justified) of the domain of $D$ on the left and the domain of $\mathcal{W}D$ on the right. To the $i^{th}$ row, we associate the function $D_i$ on the left and $\big(\mathcal{W}D\big)_i$ on the right.}
  \label{fig:Uhat}
\end{figure}

Consider any  functions $D_1,\ldots, D_n:\mathbb{Z}_{\geq 1}\to (0,\infty)$ with the convention that $D_i(0)\equiv 1$, and write $D=(D_1,\ldots, D_n)$. Treating a multipath $\pi$ as the union of all points along the various $\pi_i$, we define the partition function from $U$ to $V$ with respect to $D$ to be
$$
D\big[U\to V\big] = \sum_{\pi\in U\to V} \prod_{(x,m)\in \pi} d_{x,m}\qquad \textrm{where}\quad d_{x,m}=\frac{D_m(x)}{D_m(x-1)}.
$$

For $1\leq i\leq n$ define  $(\mathcal{W}D)_i:\mathbb{Z}_{\geq i}\to (0,\infty)$ so that for all $N\in \mathbb{Z}_{\geq 1}$ and  $1\leq \ell\leq n\wedge N$,
$$
\prod_{r=1}^\ell \big(\mathcal{W}D\big)_r(N) = D\Big[ \big(\llbracket 1,\ell\rrbracket,n\big)\to \big(\llbracket N-\ell+1,N\rrbracket,1\big)\Big],
$$
where $\big(\llbracket a,b\rrbracket,m\big) := \big((a,m),\ldots, (b,m)\big)$. Write $\mathcal{W}D = \big((\mathcal{W}D)_1,\ldots, (\mathcal{W}D)_1\big)$. (Note that in Section  \ref{sec:dpart} we define $\mathcal{W}$ in a different manner via discrete geometric Pitman transforms and it is not until Corollary \ref{cor:Wpart} that we relate $\mathcal{W}$ to the above right-hand side.)

Finally, define $\ua U$ as in Figure \ref{fig:Uhat} by lifting up the points in $U$ until they are in the domain of definition for the $\mathcal{W}D$. We may now state our main theorem.

\medskip
\begin{thm}[Theorem \ref{thm:dinv}]
$D\big[U\to V\big] = \big(\mathcal{W}D\big)\big[\ua U\to V\big]$.
\end{thm}

\medskip
Section \ref{sec:dRSK} shows that $D\mapsto \mathcal{W}D$ is half of the gRSK correspondence. {\bf  Thus, up to the action of the $\ua$ operator on $U$, the partition functions for multipaths are invariant under the gRSK correspondence}. Remark \ref{rem:newchallenges} describes the new complexities encountered in formulating and proving this theorem, most of which come from the discrete setup.

\medskip
After initially posting this paper we learned of three alternative proofs of our main result, Theorem \ref{thm:dinv}. The first is essentially already present in the work of Noumi and Yamada \cite{NY}. We summarize this in Section \ref{sec:ny}. The second is due to Konstantin Matveev and relies on the Desnanot-Jacobi identity for matrix minors. The proof was proposed after the author gave a talk on this subject in the NYC Integrable Probability Seminar. With Matveev's permission, we record this in Section \ref{sec:dj}. The third will appear in forthcoming work of Duncan Dauvergne \cite{DD} (which also contains some extensions of the results of \cite{BGW}) and relies on other manipulations involving matrix minors.

None of these three proofs were apparent (or have particularly clear analogs) in the  zero-temperature semi-discrete setting as in \cite{DOV}. By lifting to the geometric and discrete setting, we gain new tools, namely we can relate partition functions to matrix minors, and then make use of various manipulations based on that. In the zero-temperature limit (which relates to the usual RSK correspondence) one loses tools like Lindstr\"{o}m-Gessel-Viennot, and in the semi-discrete limit, one loses the matrices entirely. That is not to say that these other proofs cannot be implemented in the various degenerations, and figuring out how that works seems like a valuable problem to consider.

\smallskip
\noindent {\it Outline.}
Section \ref{sec:d} focuses on discrete polymers. The pertinent definitions and Theorem \ref{thm:dinv} are given in Section \ref{sec:dpart}, along with the proof of the theorem. Section \ref{sec:dRSK} relates the operator $\mathcal{W}$ and the discrete geometric analogs of the Pitman transform to the gRSK correspondence and row insertion. Section \ref{sec:dzero} contains the zero-temperature limit of Theorem \ref{thm:dinv}. This relates to the usual RSK correspondence and last passage percolation. Section \ref{sec:sd} contains analogous results for the semi-discrete version of the gRSK correspondence. In that setting, the proof Theorem \ref{thm:sdinv} (which is the semi-discrete analog of Theorem \ref{thm:dinv}) is considerably simpler. A reader may benefit from browsing that proof first before the proof of Theorem \ref{thm:dinv}. Section \ref{sec:ques} closes with three brief remarks/questions. The first relates to braid relations for the discrete (geometric) Pitman transforms, the second considers a KPZ-equation analog of the Airy sheet construction, and the third notes some other polymer invariances recently proved and queries whether our main theorem has a lifting to the level of stochastic vertex models.

\smallskip
\noindent {\it Notation.}
For an integer $m\in \mathbb{Z}$, we write $\mathbb{Z}_{\geq m}$ to represent all integers $n\geq m$. For two integers $m\leq n$, define $\llbracket m,n\rrbracket :=\{m,\ldots, n\}$ to be the set of all integers between $m$ and $n$, inclusive of the endpoints. For two real numbers $x$ and $y$, we use shorthand $x\wedge y=\min(x,y)$ and $x\vee y=\max(x,y)$. Other notation will be introduced where it is used.

\smallskip
\noindent {\it Acknowledgements.}
I. Corwin wishes to thank Duncan Dauvergne for graciously informing him about the work in progress \cite{DD} and pointing him to reread certain parts of the work of Noumi and Yamada \cite{NY}; Konstantin Matveev for providing an alternative third proof of Theorem \ref{thm:dinv} and granting him permission to produce that proof herein; and  Alan Hammond, Neil O'Connell, Leonid Petrov, Mark Rychnovsky and Xuan Wu for  valuable comments on a first version of this article.
I. Corwin was partially supported by the NSF grants DMS:1811143 and DMS:1664650, and the Packard Fellowship for Science and Engineering.

\section{Discrete polymers}\label{sec:d}
\subsection{Partition function invariance}\label{sec:dpart}
Before stating our main result, Theorem \ref{thm:dinv}, we first introduce non-intersecting (multi)paths, partition functions and the Pitman transform.
\begin{defn}[Non-intersecting (multi)paths]
Fix $n\in\mathbb{Z}_{\geq 2}$. We will consider the graph with vertices $\Lambda_n:=\mathbb{Z}_{\geq 1}\times\llbracket 1,n \rrbracket$ and edges between nearest-neighbors. We will call the second coordinate in $\Lambda_n$ the {\it level} and will plot the points $(x,\ell)\in \Lambda_n$ as in Figure \ref{fig:grid} so that level $\ell=1$ is on the top and level $\ell=n$ is on the bottom. In some of our discussion in what follows, we will consider subsets of $\Lambda_n$ where vertices (and also the edges incident to them) have been removed from the bottom right corner of $\Lambda$ in such a way that the removed vertices form a Young diagram under the French convention. The path and partition function definitions that follow readily generalize to that setting.

A {\it path} in $\Lambda_n$ is a finite ordered sequence of nearest neighbor, up-right connected points $\pi = \big((a_1,b_1),\ldots, (a_{L},b_{L})\big)$ in $\Lambda_n$. Precisely, for any $L\in \mathbb{Z}_{\geq 1}$ we assume that $(a_i,b_i)\in \Lambda_n$ for $i\in \llbracket 1,\ldots L\rrbracket$ and further we assume that for all $i\in \llbracket 1,L-1\rrbracket$ either
\begin{itemize}
\item $a_i-a_{i+1}=1$ and $b_i=b_{i+1}$ or
\item $a_i=a_{i+1}$ and  $b_{i+1}-b_i=1$.
\end{itemize}
The {\it starting point} of such a path is $(a_1,b_1)$ and the {\it ending point} is $(a_L,b_L)$.
We will generally use $\pi$ to denote a path and we will not label the ordered vertices traversed along the path (thus the $(a_i,b_i)$ notation will not be used below and was just introduced to clarify the definition of a path). We will often deal with multiple paths $\pi_1,\ldots,\pi_k$ in which case we will write $\pi=(\pi_1,\ldots,\pi_k)$ and call $\pi$ a {\it multipath}. Two paths $\pi_1$ and $\pi_2$ are called {\it non-intersecting} if, as sets, $\pi_1\cap\pi_2=\emptyset$. In words, $\pi_1$ and $\pi_2$ do not touch at any vertices in $\Lambda_n$ (including the starting and ending points). A multipath $\pi=(\pi_1,\ldots, \pi_k)$ is non-intersecting if for each $1\leq i\neq j\leq k$, $\pi_i$ and $\pi_j$ are non-intersecting. We will assume that all multipaths are non-intersecting. Figure \ref{fig:grid} depicts two non-intersecting paths (hence a multipath): $\pi_1$ from $(u_1,n)$ to $(v_1,1)$ and $\pi_2$ from $(u_2,n)$ to $(v_2,1)$.

\begin{figure}[t]
  \captionsetup{width=.8\linewidth}
  \includegraphics[width=3in]{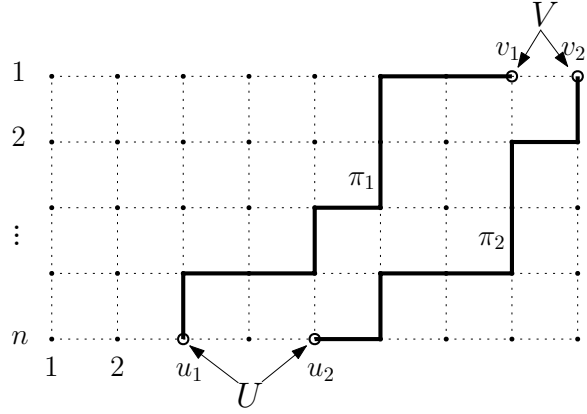}
  \caption{A subset of the lattice $\Lambda_n$ with two non-intersecting paths $\pi_1$ and $\pi_2$ depicted. The starting points $U=\big((3,n),(5,n)\big)$ and ending points $V=\big((8,1),(9,1)\big)$ (depicted by circles) constitute an endpoint pair $(U,V)$ since the set $U\to V$ of non-intersecting paths from $U$ to $V$ is non-empty.}
  \label{fig:grid}
\end{figure}

For any $k\in \mathbb{Z}_{\geq 1}$, and pair of $k$ starting and ending points
\begin{equation*}
U=\big((u_i,\ell_i)\big)_{i\in\llbracket 1,k\rrbracket}\qquad \textrm{and}\qquad V=\big((v_i,m_i)\big)_{i\in\llbracket 1,k\rrbracket}
\end{equation*}
we define the set of (non-intersecting) multipaths $\pi$ from $U$ to $V$
\begin{equation*}
U\to V := \Big\{\pi=(\pi_1,\ldots, \pi_k): \forall i\in \llbracket 1,k\rrbracket, \pi_i \textrm{ starts at }(u_i,\ell_i) \textrm{ and ends at }(v_i,m_i)\Big\}.
\end{equation*}
If the set $U\to V$ is non-empty, then we say that the pair $(U,V)$ constitute an {\it endpoint pair}.
\end{defn}

Our next definition builds upon non-intersecting paths and defines a partition function. The definition of the weight in \eqref{eq:weight} may seem a bit odd. However, if we define $d_{x,m} =  \frac{D_{m}(x)}{D_{m}(x-1)}$, then the partition functions defined below match those generally studied in the context of directed polymers with environment given by the $d_{x,m}$. We will return to this in Section \ref{sec:dRSK}.

\begin{defn}[Partition functions]\label{defn:partition}
Fix any $n$ weakly increasing integers $r_1\leq \cdots \leq r_n$ and then fix any $n$ functions $D_1,\ldots, D_n$ where $D_i:\mathbb{Z}_{\geq r_i} \to (0,\infty)$, for $i\in \llbracket 1,n\rrbracket$. We will adopt the convention that $D_i(r_i-1)=1$ for all $i\in \llbracket 1,n\rrbracket$. Write $D=(D_1,\ldots, D_n)$ to denote the $n$-tuple of such functions. We say that a point $(x,m)$ is {\it in the domain of $D$} if $m \in \llbracket 1,n\rrbracket$ and $x\geq r_{m}$. For any set $U$ of points in the domain of $D$, the {\it weight} of $U$ with respect to $D$ is
\begin{equation}\label{eq:weight}
D\big[U\big] := \prod_{(x,m)\in U} \frac{D_{m}(x)}{D_{m}(x-1)}.
\end{equation}
For a single path $\pi$ composed entirely of points in the domain of $D$, the weight of $\pi$ with respect to $D$ is given as above by treating $\pi$ as a set of points. Explicitly,
\begin{equation}\label{eq:dwt}
D\big[\pi\big] := \prod_{(x,m)\in \pi} \frac{D_{m}(x)}{D_{m}(x-1)}.
\end{equation}
For a multipath $\pi= (\pi_1,\ldots,\pi_k)$ composed of points in the domain of $D$,
\begin{equation*}
D\big[\pi\big] := D\big[\pi_1\big]\cdots D\big[\pi_k\big].
\end{equation*}
This definition is consistent with \eqref{eq:weight} if we treat $\pi$ as the union of the points in each $\pi_i$.
For any endpoint pair $(U,V)$ with both $U$ and $V$ in the domain of $D$, the {\it partition function} from $U$ to $V$ with respect to $D$ is
\begin{equation*}
D\big[U\to V\big] := \sum_{\pi\in U\to V} D\big[\pi\big].
\end{equation*}
It will also be useful to have half-open versions of the partition function:
\begin{equation}\label{eq:clopen}
D\big(U\to V\big] := \frac{D\big[U\to V\big]}{D[U]}\qquad \textrm{and} \qquad D\big[U\to V\big) := \frac{D\big[U\to V\big]}{D[V]}.
\end{equation}
\end{defn}

Our final set of definitions deal with the Pitman transform and its tensorization.
\begin{defn}[Discrete geometric Pitman transform]\label{defn:dpit}
For any $r\in \mathbb{Z}$ and any functions $f,g:\mathbb{Z}_{\geq r}\to (0,\infty)$ define functions $\big(g\odot f\big):\mathbb{Z}_{\geq r}\to (0,\infty)$ and  $\big(f\otimes g\big):\mathbb{Z}_{\geq r+1}\to (0,\infty)$ by
\begin{align}
\big(g\odot f\big)(x)&:= f(x)\cdot \sum_{m=r}^{x} \frac{g(m)}{f(m-1)} &&\textrm{for }x\geq r\\
\big(f\otimes g\big)(x)&:= g(x)\cdot \Big(\sum_{m=r}^{x} \frac{g(m)}{f(m-1)}\Big)^{-1} &&\textrm{for }x\geq r+1\\
\end{align}
where we have adopted the convention that $f(r-1)=1$ in the denominator when $m=r$. Define the operator $\Tau$ which acts on the pair $(f,g)$ as
\begin{equation}\label{eq:taudef}
\Tau\big(f,g\big) := \Big(\big(g\odot f\big),\big(f\otimes g\big)\Big).
\end{equation}
$\Tau$ acts on the tensor-product of two functions from $\mathbb{Z}_{\geq r}\to (0,\infty)$ and returns the tensor product of two functions, the first from $\mathbb{Z}_{\geq r}\to (0,\infty)$ and the second from $\mathbb{Z}_{\geq r+1}\to (0,\infty)$.

For any $m\in \llbracket 1,n-1\rrbracket$ and any $n$ weakly increasing integers $r_1\leq \cdots  \leq r_n$ with $r_m=r_{m+1}=r$ for some $r$, define the operator $\Tau_{r,m}$ which acts on the $n$-fold tensor product $D=(D_1,\ldots, D_n)$ of functions $D_i:\mathbb{Z}_{\geq r_i} \to (0,\infty)$, $i\in \llbracket 1,n\rrbracket$, as
\begin{equation}\label{eq:Taurm}
\Tau_{r,m} D := \Big(D_1,\ldots ,D_{m-1}, D_{m+1}\odot D_{m},D_{m}\otimes D_{m+1},D_{m+2},\ldots,D_{n}\Big).
\end{equation}
$\Tau_{r,m}$ acts as $\Tau$ in the $m$ and $m+1$ slot, and this action is well-defined since by assumption both $D_m,D_{m+1}:\mathbb{Z}_{\geq r}\to (0,\infty)$.

Now consider $n$ functions $D_1,\ldots, D_n:\mathbb{Z}_{\geq 1} \to (0,\infty)$ as in Definition \ref{defn:partition} and write $D=(D_1,\ldots, D_n)$. Using the operators $\Tau_{r,m}$ from \eqref{eq:Taurm}, define the operators $\mathcal{S}_r$ for $r\in \llbracket 1,n-1\rrbracket$ and the operator $\mathcal{W}$ which act on $D$ via
\begin{equation}\label{eq:SW}
\mathcal{S}_rD:=\Tau_{r,r}\Tau_{r,r+1}\cdots \Tau_{r,n-2}\Tau_{r,n-1}D,\qquad
\mathcal{W}D :=\mathcal{S}_{n-1}\mathcal{S}_{n-2}\cdots \mathcal{S}_{2} \mathcal{S}_{1}D.
\end{equation}
$\mathcal{S}_r$ acts on $D$ by first applying (in that order) $\Tau_{r,n-1}$ through $\Tau_{r,r}$. Likewise, $\mathcal{W}$ acts on $D$ by first applying (in that order) $\mathcal{S}_1$  through $\mathcal{S}_{n-1}$.
Notice that the sequence of applications of the $\Tau$ operators in $\mathcal{W}$ is well-defined. Indeed, in terms of the values of $(r_1,\ldots,r_n)$ which set the domain of $D$, initially, as assumed, $D$ has $r_i=i\wedge 1$. After applying $\Tau_{1,n-1}$ to $D$, the value of $r_{n}$ becomes $2$ (all others stay the same). Continuing in this manner, we see that for $r\in\llbracket 1,n-1\rrbracket$, $\mathcal{S}_r\cdots \mathcal{S}_1D$ has domain specified by $r_i=i\wedge (r+1)$ for $i\in \llbracket 1,n\rrbracket$. In other words, when we apply the $\Tau_{r,m}$ operator, the domain changes from having $r_m=r_{m+1}=r$ to having $r_{m}=r$ and $r_{m+1}=r+1$. Figure \ref{fig:Uhat} shows how the original domain for $D$ transforms into the domain for $\mathcal{W}D$. There is a graphical depiction of this sequence of applications of $\Tau$ which is discussed and used in Section \ref{sec:dRSK}. Also therein we explain how this operator $\mathcal{W}$ is related to the geometric RSK correspondence.
\end{defn}

We are now ready to state the main result of this paper. There is a semi-discrete version of this which can be found as Theorem \ref{thm:sdinv} in Section \ref{sec:sdpart}. Also, there are zero-temperature versions of this result which are stated as Corollary \ref{sec:dzero} in Section \ref{sec:dzero} in the discrete case, and  as Corollary \ref{sec:sdzero} in Section \ref{sec:sdzero} in the semi-discrete case. In fact, all of these results can be derived as a limit of our main theorem below.

\begin{thm}\label{thm:dinv}
Let $U=\big((u_i,n)\big)_{i\in\llbracket 1,k\rrbracket}$ and $V=\big((v_i,1)\big)_{i\in\llbracket 1,k\rrbracket}$ be any endpoint pair and define $\ua U := \big((u_i,n\wedge u_i)\big)_{i\in\llbracket 1,k\rrbracket}$ (see Figure \ref{fig:Uhat} for an illustration). Then, for any $n$ functions $D_1,\ldots, D_n:\mathbb{Z}_{\geq 1}\to (0,\infty)$, writing $D=(D_1,\ldots, D_n)$ we have
\begin{equation}\label{eq:dinv}
D\big[U\to V\big] = \big(\mathcal{W}D\big)\big[\ua U\to V\big]
\end{equation}
where the operator $\mathcal{W}$ is defined in \eqref{eq:SW}.
\end{thm}

\begin{rem}\label{rem:newchallenges}
We follow the same three step program as used to prove \cite[Proposition 4.1]{DOV} (stated here as Corollary \ref{sec:sdzero} in Section \ref{sec:sdzero}), though encounter some new complexities due to the discreteness of our present setting.
In Step 1 we consider $n=2$ and $k=1$ where $U=(u,2)$ and $V=(v,1)$ (note that for simpler notation in the $k=1$ case we drop the extra parenthesis since really $U=\big((u,2)\big)$ and likewise $V=\big((v,1)\big)$). The proof of this special case reduces to a summation identity which we readily checked by induction. Step 2 extends the result to $n=2$ and $k\geq 1$ by a decomposition of $D\big[U\to V\big]$ into a product over terms to which the result of Step 1 can be applied. Step 3 extends the result to $n\geq 2$ and $k\geq 1$ via a decomposition of $D\big[U\to V\big]$ into a sum of products of terms to which the result from Step 2 can be applied. Besides needing to prove a new summation identity in Step 1 (which is rather simple), the most significant new complexity in this proof arises from the fact that the operators $\Tau_{r,m}$ shift the domain of the functions upon which they act. Keeping track of this effect and constructing a suitable decomposition which works in these deformed domains constitutes the main challenge of the proof (which mainly arises in Step 3).

Besides proving the above theorem, it was not immediately obvious how to correctly formulate the discrete geometric generalization of \cite[Proposition 4.1]{DOV}. Two things provided some guidance in this regard. The first was a set of online notes by Y. Pei available at \url{https://toywiki.xyz/} in which the discrete non-geometric Pitman transform was discussed. The second was the work of \cite{COSZ} on the geometric RSK correspondence. Indeed, in Section \ref{sec:dRSK} we explain how the discrete geometric Pitman transform and $\mathcal{W}$ operator are related to the geometric RSK correspondence.
\end{rem}

\begin{defn}\label{def:ua}
Theorem \ref{thm:dinv} introduces the notation $\ua U$. In the proof we will need a refinement of it that we present here. For any integers $r$ and $m$, define the operator $\ua_{r,m}$ which acts on sets $U$ of points $U=\big((u_i,\ell_i)\big)_{i\in\llbracket 1,k\rrbracket}$ (for $k$ arbitrary) as follows: if $(r,m+1)\in U$, then $\ua_{r,m}U$ is composed of all points in $U$ except that $(r,m+1)$ is replaced by $(r,m)$ (if $(r,m)\in U$ as well, then $\ua_{r,m}U$ is $U$ with the point $(r,m+1)$ removed); if $(r,m+1)\notin U$, then $\ua_{r,m}U=U$. In a similar spirit to \eqref{eq:SW}, we define
\begin{equation}\label{eq:UA}
\ua_r:=\ua_{r,r}\ua_{r,r+1}\cdots \ua_{r,n-2}\ua_{r,n-1},\qquad
\ua :=\ua_{n-1}\ua_{n-2}\cdots \ua_2 \ua_{1}.
\end{equation}
The $\ua U$ defined in \eqref{eq:UA} matches $\ua U$ defined in the statement of Theorem \ref{thm:dinv}.
\end{defn}

\begin{proof}[Proof of Theorem \ref{thm:dinv}] We prove this in three steps.

\smallskip
\noindent {\bf Step 1 ($n=2$, $k=1$ case of \eqref{eq:dinv}).}
In this case, $D=(D_1,D_2)$, $U=(u,2)$ and $V=(v,1)$ (actually, $U=\big((u,2)\big)$ and $V=\big((v,1)\big)$, but we will drop the outer parenthesis in this case). Having in mind our use of this step later in Step 3 (where we deal with $n\geq 2$), we may consider a slightly more general situation where, for some $r\geq 1$, $D_1,D_2:\mathbb{Z}_{\geq r}$. In that case, what we seek to prove that for any $r\leq u\leq v$,
\begin{equation}\label{eq:dinvr}
D\big[(u,2)\to (v,1)\big] = \big(\Tau_{r,1} D\big)\big[\big(u,2\wedge(u-r+1)\big)\to (v,1)\big].
\end{equation}
Note that $\big(u,2\wedge(u-r+1)\big) = \ua_{r,1}U$. It suffices to prove \eqref{eq:dinvr} with $r=1$ since everything can be trivially shifted into larger $r$. So, for the rest of this step, let us prove \eqref{eq:dinvr} under the assumption that $r=1$. In that case $\big(u,2\wedge(u-r+1)\big) = \ua_{1,1} U=\ua U$.

The left-hand side of \eqref{eq:dinvr} can be written explicitly as
\begin{equation}\label{eq:ktwolhs}
D\big[(u,2)\to(v,1)\big]  = \frac{D_1(v)}{D_2(u-1)}\sum_{m=u}^{v} \frac{D_2(m)}{D_1(m-1)}
\end{equation}
where we have used the earlier assumed convention that $D_i(0)=1$ for $i=1$ and $2$.

Clearly there are two cases to consider, when $u=1$ and when $u>1$. When $u=1$, the identity we seek to prove reads (writing $\Tau$ in place of $\Tau_{1,1}$)
\begin{equation}\label{eq:DWD}
D\big[(1,2)\to(v,1)\big] = \big(\Tau D\big)\big[(1,1)\to (v,1)\big].
\end{equation}
Since there is only one path from $(1,1)$ to $(v,1)$, the right-hand side of \eqref{eq:DWD} equals the weight of the path from $(1,1)$ to $(v,1)$ with respect to $\Tau D$. Observing the telescoping of the product in \eqref{eq:dwt} and using the assumed convention that $(\Tau D)_1(0)=1$, we find that the right-hand side of \eqref{eq:DWD} equals $(\Tau D)_1(v)$. It follows from Definition \ref{defn:dpit} that $(\Tau  D)_1(v) = \big(D_2\odot D_1\big)(v)$ which equals the expression in \eqref{eq:ktwolhs}, thus proving the case $u=1$.

The case $x>1$ is just a bit harder. Now $\ua U=U=(u,2)$ so, using \eqref{eq:ktwolhs}, we seek to prove
\begin{equation}\label{eq:seekDWD}
\frac{D_1(v)}{D_2(u-1)}\sum_{m=u}^{v} \frac{D_2(m)}{D_1(m-1)} = \frac{\big(\Tau D\big)_1(v)}{\big(\Tau D\big)_2(u-1)}\sum_{m=u}^{v} \frac{\big(\Tau D\big)_2(m)}{\big(\Tau D\big)_1(m-1)}.
\end{equation}
Keeping in mind our convention that $D_i(0)=1$, we introduce the notation
\begin{equation*}
G_{m}:= \frac{D_2(m)}{D_1(m-1)}\qquad \textrm{and}\qquad G_{a,b}:=\sum_{m=a}^{b}G_{m}\quad \textrm{for integers }1\leq a<b.
\end{equation*}
We can rewrite $\big(\Tau D\big)_1$ and $\big(\Tau D\big)_2$ via this notation as:
\begin{equation*}
\big(\Tau D\big)_1(v) = D_1(v)G_{1,v}\qquad \textrm{and}\qquad \big(\Tau D\big)_2(u-1) = D_2(u-1)\big(G_{1,u-1}\big)^{-1}.
\end{equation*}
Using this and canceling the common factor $\frac{D_1(v)}{D_2(u-1)}$ from both sides, \eqref{eq:seekDWD} reduces to
\begin{equation*}
G_{u,v} = G_{1,v}G_{1,u-1} \sum_{m=u}^v \frac{G_m}{G_{1,m}G_{1,m-1}},
\end{equation*}
or equivalently (after moving the $ G_{1,v}G_{1,u-1}$ terms to the left-hand side)
\begin{equation}\label{eq:exactdiscide}
\frac{G_{u,v}}{G_{1,v}G_{1,u-1}} =  \sum_{m=u}^v \frac{G_m}{G_{1,m}G_{1,m-1}}.
\end{equation}
\eqref{eq:exactdiscide} is an exact summation identity which follows by induction in $v$. The base case $v=u$ is easily checked. Assuming the identity up to $v$, to show  \eqref{eq:exactdiscide} for $v+1$ it suffices to verify that
\begin{equation*}
\frac{G_{u,v+1}}{G_{1,v+1}G_{1,u-1}}  = \frac{G_{u,v}}{G_{1,v}G_{1,u-1}}  + \frac{G_{v+1}}{G_{1,v+1}G_{1,v}}.
\end{equation*}
Decomposing all terms into $G_{1,u-1}$, $G_{u,v}$ and $G_{v+1}$ (i.e., $G_{u,v+1} = G_{u,v}+G_{v+1}$, $G_{1,v+1}=  G_{1,u-1}+G_{u,v}+G_{v+1}$, and $G_{1,v} = G_{1,u-1}+G_{u,v}$) and cross-multiplying immediately yields the eqaulity. This completes the proof of \eqref{eq:exactdiscide} and hence also the proof of \eqref{eq:seekDWD} and Step 1.

\smallskip
\noindent {\bf Step 2 ($n=2$, $k\geq 1$ case of \eqref{eq:dinv}).}
Now $D=(D_1,D_2)$,  $U=\big((u_i,2)\big)_{i\in\llbracket 1,k\rrbracket}$ and $V=\big((v_i,1)\big)_{i\in\llbracket 1,k\rrbracket}$. As in Step 1, we will consider $D_1,D_2:\mathbb{Z}_{\geq r}\to (0,\infty)$ for some $r\geq 1$, and assume that all $u_i,v_i\geq r$. In this step we seek to prove that
\begin{equation}\label{eq:step2id}
D\big[U\to V\big]=\big(\Tau_{r,1}D\big)\big[\ua_{r,1} U \to V\big].
\end{equation}
Recall that $\ua_{r,1}$ is given in Definition \ref{def:ua}. In particular,  $\ua_{r,1} U =\big((u_i,2\wedge (u_i-r+1))\big)_{i\in\llbracket 1,k\rrbracket}$. As in Step 1, it suffices to prove \eqref{eq:step2id} for $r=1$, in which case $\ua_{1,1} U =\ua U$. For the rest of the proof we assume $r=1$. In that case \eqref{eq:step2id} involves $\Tau_{1,1}$ which we abbreviate  as $\Tau$.

The proof of \eqref{eq:step2id} is based on a decomposition of $D\big[U\to V\big]$ and $\big(\Tau D\big)\big[\ua U\to V\big]$ into factors which can be matched by applying the $k=1$ result from Step 1. We start by partitioning the integers $\llbracket u_1,v_k\rrbracket$ into disjoint intervals that we call type 0, type 1 or type 2. With the convention that $v_{0}=-\infty$ and $u_{k+1}=+\infty$, for each $i\in \llbracket 1,k\rrbracket$, let $a_i = u_i\vee (v_{i-1}+1)$ and $b_i=(u_{i+1}-1)\wedge v_i$. If $a_i<b_i$, then we say that  the interval $\llbracket a_i,b_i\rrbracket$ is of type $1$, and we write $\type_1$ to be the set of all such type $1$ intervals. Likewise, for each $i\in \llbracket 1,k\rrbracket$, any integer $c$ such that $y_{i-1}<c<x_{i}$ will be called type $0$. All integers which are not type $0$ and are not in intervals of type $1$ will be called type $2$ and we write $\type_2$ for the set of all such type $2$ integers. Figure \ref{fig:step2} provides an illustration for this partitioning.

%
\begin{figure}[t]
  \captionsetup{width=.8\linewidth}
  \includegraphics[width=5in]{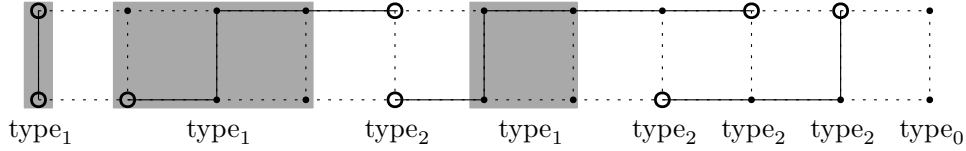}
  \caption{$U$ is represented by the open discs on the bottom level and $V$ on the top level. Intervals of type $1$ are shaded in grey; integers of types $0$ and $2$ are also labeled, but not shaded. Notice that the first type $1$ interval is actually composed of just a single integer, while the second one contains three consecutive integers and the third one contains two consecutive integers. Each type $1$ interval represents a place where multipaths $\pi\in U\to V$ may vary in terms of which points are included in $\pi$.}
  \label{fig:step2}
\end{figure}

Using the decomposition into types, we can write
\begin{equation}\label{eq:typefactor}
D\big[U\to V\big] = \prod_{\llbracket a,b \rrbracket\in \type_1} D\big[(a,2)\to (b,1)\big] \prod_{c\in\type_2} D\big[(c,2)\to (c,1)\big].
\end{equation}
To verify this note that each type 1 term can be expanded as a sum of single-path weights over paths from the bottom left to top right of the interval in question (type $1$ intervals which are composed of a single integer correspond to a single path which goes straight up from the bottom layer to the top layer). All type $2$ terms contribute a product of weights $D\big[(c,1)\big]D\big[(c,2)\big]$, and all type $0$ terms contribute a multiplicative factor of 1. This expansion shows that the right-hand side of \eqref{eq:typefactor} can be written as a sum over certain subsets of vertices of products of the weights of the vertices. By the construction of the three types, it is easy to see that these subsets are precisely in bijective correspondence with the set of multipaths in $U\to V$, hence proving \eqref{eq:typefactor}. See Figure \ref{fig:step2} for an illustration.

Given the decomposition in \eqref{eq:typefactor}, we may now apply the result of Step 1 to each term. In the product over $\type_1$ terms, note that the result of Step 1 implies that
\begin{equation*}
D\big[(a,2)\to (b,1)\big] = \big(\Tau D\big)\big[(a,2\wedge a)\to (b,1)\big].
\end{equation*}
Thus, we can replace each term in that product by the corresponding term involving $\Tau D$. For the terms in the $\type_2$ product, we also have that
\begin{equation*}
D\big[(c,2)\to (c,1)\big]= \big(\Tau D\big)\big[(c,2\wedge c)\to (c,1)\big].
\end{equation*}
This can be seen quite directly, or as a simple application of the result of Step 1 when $a=b$. Putting this all together we find that the right-hand side of \eqref{eq:typefactor} equals
\begin{equation*}
\prod_{\llbracket a,b\rrbracket\in \type_1} \big(\Tau D\big)\big[(a,2\wedge a)\to (b,1)\big] \prod_{c \in\type_2} \big(\Tau D\big)\big[(c,2\wedge c)\to (c,1)\big] = \big(\Tau D\big)\big[\ua U\to V\big],
\end{equation*}
The equality above is a similar decomposition as in \eqref{eq:typefactor}, the main difference being that $U$ is replaced by $\ua U$. This completes the proof of \eqref{eq:step2id} (with $r=1$) and hence Step 2.

\smallskip
\noindent {\bf Step 3 ($n\geq 2, k\geq 1$ case of \eqref{eq:dinv}).}
In this case, $D=(D_1,\ldots, D_n)$, $U=\big((u_i,n)\big)_{i\in\llbracket 1,k\rrbracket}$ and $V=\big((v_i,1)\big)_{i\in\llbracket 1,k\rrbracket}$. The key to this proof is a decomposition of $D\big[U\to V\big]$ into a sum of products to which we can apply the $n=2$ result proved earlier in Step 2. Since this decomposition is a bit involved, we will first explicitly work out the $n=3$ proof in that hope that it may make it easier to understand the general $n$ proof.

Before embarking on the $n=3$ case  it will be useful to rewrite the identity \eqref{eq:step2id} proved in Step 2 in terms of the notation that we will use below. Consider any $m\in \llbracket 1,n-1\rrbracket$ and any sequence of weakly increasing integers $r_1\leq \cdots \leq r_n$ such that $r_{m}=r_{m+1}=r$ for some $r$. Further, consider any $n$ functions $\tilde D=(\tilde D_1,\ldots, \tilde D_n)$ where  $\tilde D_i:\mathbb{Z}_{\geq r_i} \to (0,\infty)$ for $i\in \llbracket 1,n\rrbracket$, and any endpoint pair $(W,Z)$ where $W=\big((w_i,m+1)\big)_{i\in \llbracket 1,k\rrbracket}$, $Z=\big((z_i,m)\big)_{i\in \llbracket 1,k\rrbracket}$ and all $w_i,z_i\geq r$. Then, \eqref{eq:step2id} of Step 2 implies that
\begin{equation}\label{eq:step2idredux}
\tilde D\big[W\to Z\big]=\big(\Tau_{r,m} \tilde D\big)\big[\ua_{r,m} W\to Z\big].
\end{equation}

We return now to proving \eqref{eq:dinv} with $n=3$. We will refer to Figure \ref{fig:threelevel} in describing certain steps the proof. Recall that for $n=3$, $\mathcal{W} = \mathcal{S}_2 \mathcal{S}_1 = \Tau_{2,2}\Tau_{1,1}\Tau_{1,2}$. Thus, in order to prove the equality in \eqref{eq:dinv}, we will prove three intermediate equalities, relating an expression with $D$ to one with $\Tau_{1,2}D$, then relating that to an expression with $\Tau_{1,1}\Tau_{1,2}D$ and finally relating that to an expression with $\Tau_{2,2}\Tau_{1,1}\Tau_{1,2}D$. Collecting these equalities together will prove \eqref{eq:dinv}.

We start with the following decomposition (recall half-open partition functions from \eqref{eq:clopen})
\begin{equation}\label{eq:UVlongdecom}
D\big[U\to V\big] = \sum_{Z} D\big[U\to W^+)D\big[W^+\to W) D\big[W\to Z\big] D\big(Z\to Z^-] D\big(Z^-\to V],
\end{equation}
with notation that we now define.
The summation in \eqref{eq:UVlongdecom} is over $Z= \big((z_i,2)\big)_{i\in\llbracket 1,k\rrbracket}$.
Given $Z$, we define $Z^-:= \big((z_i,1)\big)_{i\in\llbracket 1,k\rrbracket}$ (i.e., points in $Z^-$ match those in $Z$ except they have a level one less).
The $W$ and $W^+$ are simply set equal to $U$. (It may seem unnecessary to include $W$ and $W^+$ here, and it is. However, in the general $n$ case they will be important, so we keep them to match with the notation eventually used there.)
The summation in \eqref{eq:UVlongdecom} is restricted to only those $Z$ such that $(U,W^+)$, $(W^+,W)$, $(W,Z)$, $(Z,Z^-)$ and $(Z^-,V)$ are all endpoint pairs.
Note that due to the assumption that $W=W^+=U$ and the relationship between $Z$ and $Z^-$, these five conditions reduce to two: that $(W,Z)$ and $(Z^-,V)$ are endpoint pairs.
Furthermore, the summand in the right-hand side of \eqref{eq:UVlongdecom} reduces to $D\big[U\to Z\big] D\big(Z\to Z^-] D\big(Z^-\to V]$. Display box $1$ of Figure \ref{fig:threelevel} illustrates the choices of $W,W^+,Z,Z^-$ in this decomposition.

\begin{figure}[t]
  \captionsetup{width=.8\linewidth}
  \includegraphics[width=6in]{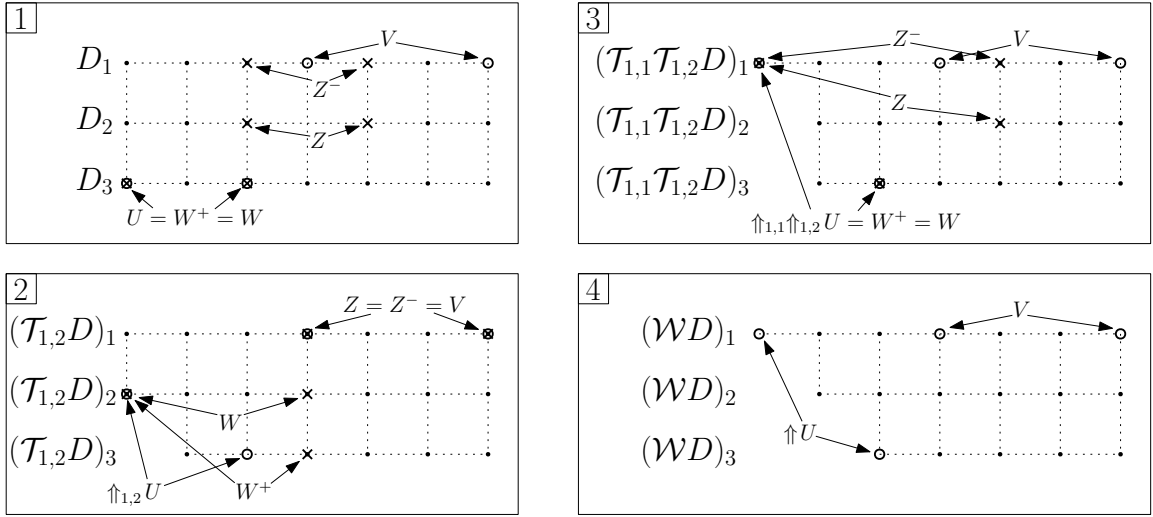}
  \caption{Sets of points used in the proof of (\ref{eq:dinv}) when $n=3$ and $k=2$.}
  \label{fig:threelevel}
\end{figure}

To see why the decomposition in \eqref{eq:UVlongdecom} holds, recall that $D\big[U\to V\big]$ is equal to a sum of weights of multipaths $\pi$ from $U$ to $V$. We can partition that sum based on the locations  $Z$ of the exit points for the multipath $\pi$ going from level 2 to level 1 (hence $Z^-$ has the meaning as the entry point locations into level 1). The sum over weights of multipaths $\pi$ which have a given value of $Z$ factorizes precisely as in \eqref{eq:UVlongdecom}, thus proving the decomposition. Notice that the use of the half-open and half-closed path sums defined in \eqref{eq:clopen} is important here to ensure that the correct weight factors arise in our decomposition.

The decomposition in \eqref{eq:UVlongdecom} has isolated the dependence on $D_2$ and $D_3$ in the right-hand side summand. Indeed, since $D\big[U\to W^+)=D\big[W^+\to W)=1$, they do not depend on $D_2$ or $D_3$, and clearly $D\big(Z\to Z^-]=D[Z^-]$  and $D\big(Z^-\to V]$ only depend on $D_1$.
This leaves $D\big[W\to Z\big]$ as the only term which involves $D_2$ and $D_3$.
We now apply \eqref{eq:step2idredux} with $r=1$, $m=2$, and $\tilde D = D$.
This implies that $D\big[W\to Z\big] =\big(\Tau_{1,2} D\big)\big[\ua_{1,2} W\to Z\big]$.
Since $\big(\Tau_{1,2}D\big)_1=D_1$, we may replace $D$ by $D'=\big(\Tau_{1,2}D\big)$ in the other terms in the summand of \eqref{eq:UVlongdecom} without changing their values (they are all either constant or only depend on $D_1$).
Moreover, since $D\big[U\to W^+)=D\big[W^+\to W)=1$, we may replace these terms by $D\big[\ua_{1,2}U\to \ua_{1,2} W^+)=D\big[\ua_{1,2} W^+\to \ua_{1,2} W)=1$.
This implies that, denoting $D'= \Tau_{1,2} D$,  $U'=\ua_{1,2}U$, $W'^+=\ua_{1,2}W^+$ and $W'=\ua_{1,2}W$,
\begin{equation*}
D\big[U\to V\big] =  \sum_{Z} D'\big[U'\to  W'^+)D'\big[W'^+\to W') D'\big[W'\to Z\big] D'\big(Z\to Z^-] D'\big(Z^-\to V],
\end{equation*}
where we now assume that $U'= W'^+=W'$ and where the sum over $Z$ is such that $(W',Z)$ and $(Z^-,V)$ are endpoint pairs (note that the condition that $(W',Z)$ is an endpoint pair imposes the same restriction on $Z$ as the condition that $(W,Z)$ is an endpoint pair). This right-hand side is a decomposition of $D'\big[U'\to V\big]$ in the same spirit as \eqref{eq:UVlongdecom}, hence we have shown that (going back to $\Tau_{1,2}D$ in place of $D'$ and $\ua_{1,2}U$ in place of $U'$)
\begin{equation}\label{eq:UVlongdecom2}
D\big[U\to V\big] = \big(\Tau_{1,2}D\big)\big[\ua_{1,2} U \to V\big].
\end{equation}
This is the first of our three intermediate equalities.

In order to relate the right-hand side of \eqref{eq:UVlongdecom2} to an expression involving $\Tau_{1,1}\Tau_{1,2}D$, we employ another decomposition.
As above, let us denote $D'= \Tau_{1,2} D$ and $U'=\ua_{1,2}U$. Then,
\begin{equation}\label{eq:UVlongdecom3}
D'\big[U'\to V\big] = \sum_{W} D'\big[U'\to W^+)D'\big[W^+\to W) D'\big[W\to Z\big] D'\big(Z\to Z^-] D'\big(Z^-\to V],
\end{equation}
with notation that we now define. The summation in \eqref{eq:UVlongdecom2} is over $W= \big((w_i,2)\big)_{i\in\llbracket 1,k\rrbracket}$.
Given $W$, we define $W^+=\big((w_i,(w_i+1)\wedge 3)\big)_{i\in\llbracket 1,k\rrbracket}$ (i.e., as long as the first coordinate is $\geq 2$, every point in $W$ corresponds to a point in $W^+$ with one larger level; if $(1,2)\in W$ then $(1,2)\in W^+$ as well).
The $Z$ and $Z^-$ are simply set equal to $V$.
The summation in \eqref{eq:UVlongdecom2} is restricted to only those $W$ such that $('U,W^+)$, $(W^+,W)$, $(W,Z)$, $(Z,Z^-)$ and $(Z^-,V)$ are all endpoint pairs.
Note that due to the assumption that $Z=Z^-=V$ and the relationship between $W$ and $W+$, these five conditions reduce to two: that $('U,W^+)$ and $(W,Z)$ are endpoint pairs.
Furthermore, the summand in the right-hand side of \eqref{eq:UVlongdecom2} reduces to $D'\big[U'\to W^+)D'\big[W^+\to W) D'\big[W\to Z\big]$. Display box $2$ of Figure \ref{fig:threelevel} illustrates the choices of $W,W^+,Z,Z^-$ in this decomposition. The same reasoning as for the first decomposition \eqref{eq:UVlongdecom} applies in justifying its validity.

Now, notice that the decomposition in \eqref{eq:UVlongdecom2} has isolated the dependence on $D'_1$ and $D'_2$ in the right-hand side summand. Indeed, since $D'\big(Z\to Z^-]= D'\big(Z^-\to V]=1$, they do not depend on $D'_1$ or $D'_2$, and clearly $D'\big[U'\to W^+)$ and $D'\big[W^+\to W)$ only depend on $D'_3$ (remember the notation defined in \eqref{eq:clopen} to see this).
This leaves $D'\big[W\to Z\big]$ as the only term which involves $D'_1$ and $D'_2$.
We now apply \eqref{eq:step2idredux} with $r=1$, $m=1$, and $\tilde D = D'$.
This implies that $D'\big[W\to Z\big] =\big(\Tau_{1,1} D'\big)\big[\ua_{1,1} W\to Z\big]$.
Since $\big(\Tau_{1,1}D'\big)_3=D'_3$, we may replace $D'$ by $D''=\big(\Tau_{1,1}D'\big)$ in the other terms in the summand of \eqref{eq:UVlongdecom2} without changing their values (they are all either constant or only depend on $D'_3$).
Now observe that
\begin{equation*}
D''\big[U'\to W^+)=D''\big[\ua_{1,1}U'\to \ua_{1,1}W^+) \quad \textrm{and} \quad
D''\big[W^+\to W)=D''\big[\ua_{1,1}W^+\to \ua_{1,1}W),
\end{equation*}
and that the condition that  $(U',W^+)$ and $(W,Z)$ are endpoint pairs is equivalent to the condition that  $(\ua_{1,1} U,\ua_{1,1}W^+)$ and $(\ua_{1,1}W,Z)$ are endpoint pairs. Thus, writing $U''=\ua_{1,1}U'$, $W'^+=\ua_{1,1}W^+$ and $W'=\ua_{1,1}W$
we have shown that
\begin{equation}\label{eq:UVlongdecom33}
D'\big[U'\to V\big] = \sum_{W'} D''\big[U''\to W'^+)D''\big[W'^+\to W') D''\big[W'\to Z\big] D''\big(Z\to Z^-] D''\big(Z^-\to V],
\end{equation}
where $Z=Z^-=V$ and where the sum is restricted to $W'$ such that $(U'',W'^+)$ and $(W',Z)$ are endpoint pairs. This, however, is a decomposition for $D''\big[U''\to V\big]$, and hence we have shown that (going back to $\Tau_{1,1}\Tau_{1,2}D$ in place of $D''$ and $\ua_{1,1}\ua_{1,2}U$ in place of $U''$)
\begin{equation}\label{eq:UVlongdecom4}
\big(\Tau_{1,2}D\big)\big[\ua_{1,2}U\to V\big] = \big(\Tau_{1,1}\Tau_{1,2}D\big)\big[\ua_{1,1}\ua_{1,2} U \to V\big].
\end{equation}
This is the second of our three intermediate equalities.

The final step is to use a decomposition similar to that of \eqref{eq:UVlongdecom} to relate the right-hand side of \eqref{eq:UVlongdecom4} above to an expression involving $\Tau_{2,2}\Tau_{1,1}\Tau_{1,2}D$. As above, let us denote $D''=\Tau_{1,1}\Tau_{1,2}D$ and $U''=\ua_{1,1}\ua_{1,2} U$. Then
\begin{equation}\label{eq:UVlongdecom5}
D''\big[U''\to V\big] = \sum_{Z} D''\big[U''\to W^+)D''\big[W^+\to W) D''\big[W\to Z\big] D''\big(Z\to Z^-] D''\big(Z^-\to V],
\end{equation}
with notation that we now define. The summation is over sets of $k$ points $Z$ which are either on level $2$ or the point $(1,1)$. For each point on level 2 in $Z$, there is a point on level 1 in $Z^-$ with the same first coordinate; and if $(1,1)\in Z$ then $(1,1)\in Z^-$ as well. The
$W$ and $W^+$ are simply set to be $U''$.
The summation in \eqref{eq:UVlongdecom5} is over $Z$ such that $(U,W^+)$, $(W^+,W)$, $(W,Z)$, $(Z,Z^-)$ and $(Z^-,V)$ are all endpoint pairs.
Note that due to the assumption that $U''=W^+=W$ and the relationship between $Z$ and $Z^-$, these five conditions reduce to two: that $(W,Z)$ and $(Z^-,V)$ are endpoint pairs.
Furthermore, the summand in the right-hand side of \eqref{eq:UVlongdecom} reduces to $D''\big[U''\to Z\big] D''\big(Z\to Z^-] D''\big(Z^-\to V]$. Display box $3$ of Figure \ref{fig:threelevel} illustrates the choices of $W,W^+,Z,Z^-$ in this decomposition. The same reasoning as for the first decomposition \eqref{eq:UVlongdecom} applies in justifying its validity.

The decomposition in \eqref{eq:UVlongdecom5} has isolated the dependence on $D''_2$ and $D''_3$ in the right-hand side summand. Indeed, since $D''\big[U''\to W^+)=D''\big[W^+\to W)=1$, they do not depend on $D''_2$ or $D''_3$, and clearly $D''\big(Z\to Z^-]$ and $D''\big(Z^-\to V]$ only depend on $D''_1$. This leaves $D''\big[W\to Z\big]$ as the only term which involves $D''_2$ and $D''_3$. We claim that by applying \eqref{eq:step2idredux} with $r=2$, $m=2$, and $\tilde D = D''$, we can show that
\begin{equation}\label{eq:Dpp}
D''\big[W\to Z\big] =\big(\Tau_{2,2} D''\big)\big[\ua_{2,2} W\to Z\big].
\end{equation}
 This equality requires a bit of explanation. If $(1,1)\notin W$ (i.e., if $(1,1)\notin U$) then $W$ may only include points on level $3$ and $Z$ may only include points on level $2$. In this case, we may directly apply \eqref{eq:step2idredux} to conclude the equality \eqref{eq:Dpp}. On the other hand, if $(1,1)\in W$, it must also be in $Z$. Thus, if we write $\hat{W}$ and $\hat{Z}$ for $W$ and $Z$ with the point $(1,1)$ removed, then
 $D''\big[W\to Z\big] = D''\big[(1,1)\big]D''\big[\hat{W}\to \hat{Z}\big]$. We may apply \eqref{eq:step2idredux} (with $r=2$, $m=2$, and $\tilde D = D''$) to show that $D''\big[\hat{W}\to \hat{Z}\big] = \big(\Tau_{2,2} D''\big)\big[\ua_{2,2} \hat{W}\to \hat{Z}\big]$.
 Finally, since $\big(\Tau_{2,2}D''\big)_1=D''_1$, we have that $D''\big[(1,1)\big]=\big(\Tau_{2,2}D''\big)\big[(1,1)\big]$, which proves \eqref{eq:Dpp}.

 Using the fact that $\big(\Tau_{2,2}D''\big)_1=D''_1$ again, we may replace $D''$ by $D'''=\Tau_{2,2}D''$ in all of the other terms in the summand of \eqref{eq:UVlongdecom5} without changing their values (they are all either constant or only depend on $D''_1$).
Moreover, since $D''\big[U''\to W^+)=D''\big[W^+\to W)=1$, we may replace these terms by $D'''\big[\ua_{2,2}U''\to \ua_{2,2} W^+)=D''\big[\ua_{2,2} W^+\to \ua_{2,2} W)=1$.
This implies that, denoting $D'''= \Tau_{2,2} D''$,  $U'''=\ua_{2,2}U''$, $W'^+=\ua_{2,2}W^+$ and $W'=\ua_{2,2}W$,
\begin{equation*}
D\big[U\to V\big] =  \sum_{Z} D'''\big[U'''\to  W'^+)D'''\big[W'^+\to W') D'''\big[W'\to Z\big] D'''\big(Z\to Z^-] D'''\big(Z^-\to V],
\end{equation*}
where we now assume that $U'''= W'^+=W'$ and where the sum over $Z$ is such that $(W',Z)$ and $(Z^-,V)$ are endpoint pairs (note that the condition that $(W',Z)$ is an endpoint pair imposes the same restriction on $Z$ as the condition that $(W,Z)$ is an endpoint pair). This right-hand side is a decomposition of $D'''\big[U'''\to V\big]$ in the same spirit as \eqref{eq:UVlongdecom5}, hence we have shown (going back to $\Tau_{2,2}\Tau_{1,1}\Tau_{1,2}D$ in place of $D'''$ and $\ua_{2,2}\ua_{1,1}\ua_{1,2}U$ in place of $U'''$)
\begin{equation}\label{eq:UVlongdecom6}
\big(\Tau_{1,1}\Tau_{1,2}D\big)\big[\ua_{1,1}\ua_{1,2} U \to V\big] = \big(\Tau_{2,2}\Tau_{1,1}\Tau_{1,2}D\big)\big[\ua_{2,2}\ua_{1,1}\ua_{1,2} U \to V\big].
\end{equation}

Combining together \eqref{eq:UVlongdecom2}, \eqref{eq:UVlongdecom4} and \eqref{eq:UVlongdecom6} proves \eqref{eq:dinv} for $n=3$ and general $k$.

Now we turn to address the general $n$ case. This proceeds by a sequence of $n\choose 2$ replacements through which we transition from $D$ to $\mathcal{W}D$. The following lemma is the key to this scheme. Figure \ref{fig:LemDec} may be helpful in understanding the hypotheses in the statement of the lemma.

\begin{lemma}\label{lem:inductive}
Fix any $m\in \llbracket 1,n-1\rrbracket$, any $n$ weakly increasing integers $r_1\leq \cdots  \leq r_n$ with (for some $r$) $r=r_m=r_{m+1}<r_{m+2}$ (if $m=n-1$ then by convention set $r_{n+1}=+\infty$), and any $n$ functions $D=(D_1,\ldots, D_n)$ such that $D_i:\mathbb{Z}_{\geq r_i} \to (0,\infty)$, $i\in \llbracket 1,n\rrbracket$. Fix any $k\in \mathbb{Z}_{\geq 1}$ and any $k$-tuple of points $U$ such that every point in $U$ is either of the form  $(u,n)$ for some $u\geq r_n$, or of the form $(r_i,i)$ for some $i\in \llbracket 1,n\rrbracket$ such that $r_i<r_{i+1}$. Finally, fix another $k$-tuple of points $V=\big((v_i,1)\big)_{i\in \llbracket 1,k\rrbracket}$ such that $(U,V)$ form an endpoint pair. Then
\begin{equation}\label{eq:DUVLEMMA}
D\big[U\to V\big] = \big(\Tau_{r,m}D\big)\big[\ua_{r,m} U\to V\big].
\end{equation}
\end{lemma}

Let us assume this lemma for the moment and conclude the proof of \eqref{eq:dinv}. The left-hand side of \eqref{eq:dinv} fits into the setup of Lemma \ref{lem:inductive} with $r_1=\cdots = r_{n}=1$ and hence $U=\big((u_i,n)\big)_{i\in \llbracket 1,k\rrbracket}$. Now (as prescribed by the definition of $\mathcal{W}$) sequentially apply the lemma with $r=1$ and $m=n-1$ through $1$, and then $r=2$ and $m=n-1$ through $2$, and so forth until finally $r=n-1$ and $m=n-1$. Each application amounts to applying the relevant $\Tau_{r,m}$ to the $D$ functions (or rather, to the image of the $D$ functions under previous application of $\Tau$ operators) and the relevant $\ua_{r,m}$ operator to the $U$ points (or rather, to the image of the $U$ points under previous application of $\ua$ operators). It is clear that the output of each step in this sequence is suitable for applying the next step of the lemma. Recalling the definition of $\mathcal{W}$ and $\ua$ in terms of the $\Tau_{r,m}$ and $\ua_{r,m}$, we find that this procedure shows that
\begin{equation}
D\big[U\to V\big] = \big(\mathcal{W}D\big)\big[\ua U\to V\big],
\end{equation}
precisely as claimed in \eqref{eq:dinv}. Thus, to complete this step  we are left to prove Lemma \ref{lem:inductive}.

\begin{figure}[t]
  \captionsetup{width=.8\linewidth}
  \includegraphics[width=6in]{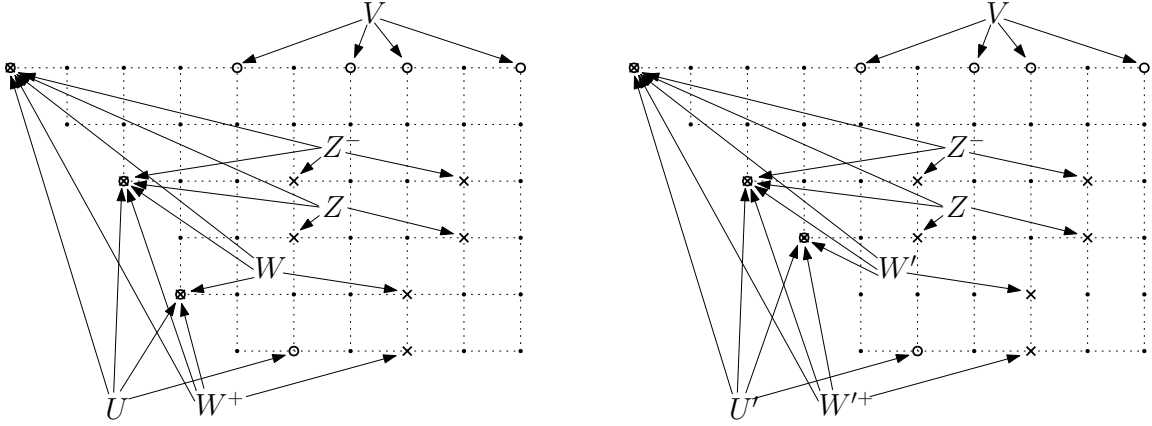}
  \caption{An illustration of the decomposition (\ref{eq:bigdecomp}) in Lemma \ref{lem:inductive}. On the left, we have $r_1=1, r_2=2, r_3=3, r_4=4, r_5=4, r_6=5$ so the domain of $D$ satisfies the hypotheses of Lemma \ref{lem:inductive} with $m=4$ and $r=4$. On the right, all of the $r_i$ remain the same except $r_5=5$ now. On the left, the starting points $U$ satisfy the hypotheses of Lemma \ref{lem:inductive}. Notice that graphically, the hypothesis on $U$ corresponds to having every point at the ``bottom'' of the domain of $D$ (no other points in the domain of $D$ below the points of $U$). Also shown in the figure are the choices of $W^+, W, Z,$ and $Z^-$ for the decomposition (\ref{eq:bigdecomp}) on the left and (\ref{eq:bigdecomp2}) on the right.}
  \label{fig:LemDec}
\end{figure}

\begin{proof}[Proof of Lemma \ref{lem:inductive}]
We appeal to the following decomposition which simultaneously captures all of the decompositions used in the $n=3$ proof given earlier. Assuming the hypotheses on the $r_i$, $U$  and $V$ from the statement of Lemma \ref{lem:inductive}, we have that
\begin{equation}\label{eq:bigdecomp}
D\big[U\to V\big] =  \sum_{W,Z} D\big[U\to  W^+)D\big[W^+\to W) D\big[W\to Z\big] D\big(Z\to Z^-] D\big(Z^-\to V],
\end{equation}
where all sets $U,W^+,W,Z,Z^-,V$ contain $k$ points and additionally satisfy
\begin{itemize}
\item $W$ contains all points in $U$ of level $\leq m+1$, and for each point in $U$ of level $>m+1$, there is a unique point in $W$ of level $m+1$;
\item $W^+$ contains all points in $U$ of level $\leq m+1$, and for each point in $U$ of level $>m+1$, there is a unique point in $W^+$ of level $m+2$;
\item $W\setminus U$ and $W^+\setminus U$ are in the following correspondence: for each $(w,m+1)\in W\setminus U$ there is a point $(w,m+2)\in W^+\setminus U$;
\item If $m=1$, then $Z=Z^-=V$.
\item If $m\geq 2$, $Z$ contains all points in $U$ of level $\leq m-1$, and for each point in $U$ of level $\geq m+1$, there is a unique point in $Z$ of level $m$;
\item If $m\geq 2$, $Z^+$ contains all points in $U$ of level $\leq m-1$, and for each point in $U$ of level $\geq m+1$, there is a unique point in $Z^+$ of level $m-1$;
\item $Z\setminus V$ and $Z^-\setminus U$ are in the following correspondence: for each $(z,m)\in Z\setminus U$ there is a point $(z,m-1)\in Z^-\setminus U$;
\item The sum over $W$ and $Z$ is restricted so as to satisfy the above conditions as well as the condition that $(U,W^+)$, $(W^+,W)$, $(W,Z)$, $(Z,Z^-)$, and $(Z^-,V)$ are all endpoint pairs.
\end{itemize}
Note that the hypotheses on the $r_i$ and $U$  from the statement of Lemma \ref{lem:inductive} imply that there are no points in $U$ of level $m$ (since we have assumed $r_m=r_{m+1}$ and all points of the form $(r_i,i)\in U$ have $r_i<r_{i+1}$). This implies that the sets above all have exactly $k$ points. \eqref{eq:bigdecomp} follows by expanding each term on the right-hand side and then verifying that the resulting summation puts each term in bijection with a multipath with the corresponding weight.

The key fact that is evident from the decomposition \eqref{eq:bigdecomp} is that the only term on the right-hand side inside the summation over $W$ and $Z$ which depends on the functions $D_m$ and $D_{m+1}$ is $D\big[W\to Z\big]$. We can use this fact along with \eqref{eq:step2idredux} (with $\tilde D=D$ and $r$ and $m$ as specified in the  statement of Lemma \ref{lem:inductive}) to show that
\begin{equation}\label{eq:WZrm}
D\big[W\to Z\big] = \big(\Tau_{r,m}D\big)\big[\ua_{r,m}W\to Z\big].
\end{equation}
Now, there are two cases to consider: when $(r,m+1)\in U$ or when $(r,m+1)\notin U$.

When $(r,m+1)\in U$, we also have $(r,m+1)\in W$ and $(r,m+1)\in W^+$. If we define $U'=\ua_{r,m}U$, $W'=\ua_{r,m}W$, $W'^+=\ua_{r,m}W^+$ and $D'=\Tau_{r,m}D$, then it follows from \eqref{eq:WZrm} and \eqref{eq:bigdecomp} that
\begin{equation}\label{eq:bigdecomp2}
D\big[U\to V\big] =  \sum_{W',Z} D'\big[U'\to  W'^+)D'\big[W'^+\to W') D'\big[W'\to Z\big] D'\big(Z\to Z^-] D'\big(Z^-\to V].
\end{equation}
We have used the fact that the only term in  \eqref{eq:bigdecomp} which depends on the functions $D_m$ and $D_{m+1}$ is $D\big[W\to Z\big]$, and hence replacing $D$ by $D'$ does not effect the other terms. We have also used the fact that since $(r,m+1)$ is common to $U, W$ and $W^+$, replacing it by $(r,m)$ does not change the value of the first two terms in the summand in \eqref{eq:bigdecomp}.

When $(r,m+1)\notin U$, we also have $(r,m+1)\notin W$ and $(r,m+1)\notin W^+$. In that case, $W'=\ua_{r,m}W=W$ and likewise $U'=\ua_{r,m}U=U$ and $W'^+=\ua_{r,m}W^+=W^+$. Thus, setting $D'=\Tau_{r,m}D$, by the same reasoning as above we find that the decomposition \eqref{eq:bigdecomp2} holds.

It remains to observe that the right-hand side of \eqref{eq:bigdecomp2} is a decomposition of $D'\big[U'\to V\big]$. The only difference relative to the decomposition in \eqref{eq:bigdecomp} is that $U$ is replaced by $U'$ and that $W'$ and $W'^+$ now contain points from $U'$ of level $\leq m$. Recalling that $D'=\Tau_{r,m}D$ and $U'=\ua_{r,m}U$ we see that we have proved \eqref{eq:DUVLEMMA} and hence the lemma.
\end{proof}

This completes the proof of the third step, and hence the proof of Theorem \ref{thm:dinv}.
\end{proof}

\subsection{Relation to geometric RSK correspondence}\label{sec:dRSK}
The operator $\mathcal{W}$ defined via the (discrete geometric) Pitman transform in \eqref{eq:SW} is closely related to the geometric RSK correspondence. We recall the geometric RSK correspondence in Section \ref{sec:recall} and state the relationship with  $\mathcal{W}$ in Section \ref{sec:RSKPit}. Finally, in Section \ref{sec:inv} we consider a special choice for the functions $D$ (related to inverse-gamma random variables) and briefly recall how, via \cite{COSZ}, this gives rise to a Markovian structure for $\big(\mathcal{W}D\big)(t)$.

\subsubsection{Recalling the geometric RSK}\label{sec:recall}
We recall the geometric RSK correspondence (gRSK) as defined in \cite[Definitions 2.1 and 2.2]{COSZ} (see also \cite{Kir,NY}).

\begin{defn}\label{NYdef}
Fix any $N\in \mathbb{Z}_{\geq 1}$.
Let $1\leq \ell\leq N$. Consider two {\it words} $\xi=(\xi_{\ell},\ldots,\xi_N)$ and $b=(b_\ell,\ldots, b_N)$ with strictly positive real entries. {\it Geometric row insertion} of the word $b$ into the word $\xi$  transforms the pair $(\xi,b)$ into a new pair $(\xi',b')$ where $\xi'=(\xi'_{\ell},\ldots, \xi'_N)$ and $b'=(b'_{\ell+1},\ldots, b'_N)$ (i.e., $b'$ has one fewer entry and the index starts at $\ell+1$ instead of $\ell$). The transformation is notated and defined as follows:
\begin{equation}
\begin{array}{ccc}
& b & \\
\xi & \cross & \xi' \\
& b'
\end{array}
\quad
\textrm{where} \quad
\begin{cases}
\xi'_\ell=b_\ell\xi_\ell, &\\[4pt]
\xi'_k =b_k(\xi_{k-1}' + \xi_k), & \ell+1\le k\le N\\[5pt]
b'_{k} = b_k \ddd\frac{ \xi_{k}\xi'_{k-1}}{\xi_{k-1}\xi'_k}, & \ell+1\leq k\leq N.
\end{cases}
\label{g-row-ins}\end{equation}
If $\ell=N$, the output word $b'$ is empty.  
In addition to $\xi\in(0,\infty)^{N-\ell+1}$ we admit the case
$\xi=(1,0,\dotsc,0)$.   This will correspond to row insertion into an initially empty word.
For $k\geq 1$ we denote such a word as  $e^{(k)}_1=(1,0,\dotsc,0)$, where $k$ is the total number of coordinates for the vector. The notation and  definition are extended so that
\begin{equation}
\begin{array}{ccc}
& b & \\
e_1^{(N-\ell+1)} & \cross & \xi'  \\[3pt]
\end{array}
\quad
\textrm{where} \quad  \xi'_{k}=\prod_{i=\ell}^{k} b_{i},  \quad \ell\le k\le N.
\label{es-row-ins}
\end{equation}
This is consistent with \eqref{g-row-ins} except that   output $b'$ is not defined and hence
not displayed in the diagram above.

For $N$ fixed, consider a semi-infinite array of strictly positive real numbers $m=\big(m_{j,n}: j\in \llbracket 1,N\rrbracket, n\geq 1\big)$. For integers $1\leq a\leq b\leq N$ and $1\leq c\leq d$, we write $m_{\llbracket a,b\rrbracket,\llbracket c,d\rrbracket}:=\big(m_{j,n}: j\in \llbracket a,b\rrbracket, n\in \llbracket a,b\rrbracket\big)$ for the corresponding subarray. In particular we will generally fix $N$ in which case we denote the $n^{th}$ row of $m$ by $m_{\llbracket 1,N\rrbracket, \llbracket n\rrbracket}=(m_{1,n},\ldots, m_{n,N}\big)$ and the first $n$ rows by $m_{\llbracket 1,N\rrbracket,\llbracket 1,n\rrbracket}$. See Figure \ref{fig:marray} for an illustration of this notation.

\begin{figure}[ht]
  \captionsetup{width=.8\linewidth}
  \includegraphics[width=2.5in]{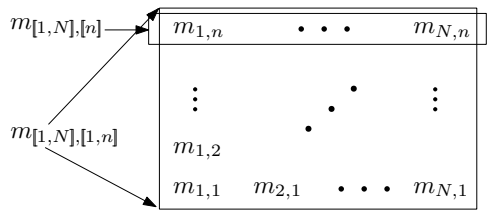}
  \caption{The first $n$ rows and $N$ columns of the array $m$.}
  \label{fig:marray}
\end{figure}

The {\it geometric RSK} correspondence is a bi-rational map between $m_{\llbracket 1,N\rrbracket,\llbracket 1,n\rrbracket}$ and another $N$ by $n$ matrix of strictly positive real numbers. As we will explain in what follows, under the gRSK, $m_{\llbracket 1,N\rrbracket,\llbracket 1,n\rrbracket}$  maps to a pair of {\it geometric} Young Tableaux  $(P,Q)$ of the same {\it shape}. Let us start by defining the $P$-geometric Young tableaux which is a map from $m_{\llbracket 1,N\rrbracket,\llbracket 1,n\rrbracket}$ to an array of the form
\begin{equation}\label{eq:zz}
z(n) = \big(z_{k,\ell}: 1\leq \ell\leq k\leq N\textrm{ and } \ell\in \llbracket 1,N\wedge n\rrbracket\big).
\end{equation}
When $n\geq N$, $z(n)$ is a triangular array, and when $n\leq N$ it is a trapezoidal array.
For $\ell\in \llbracket 1,N\wedge n\rrbracket$, denote by $z_\ell(n) = \big(z_{\ell,\ell}(n),\ldots, z_{N,\ell}(n)\big)$ the $\ell$-th diagonal and denote by $\sh(z(n))=\big(z_{N,\ell}(n)\big)_{\ell\in\llbracket 1,n\wedge N\rrbracket}$ the bottom row of $z(n)$, which we call the {\it shape} of $z(n)$.  See Figure \ref{fig:arrays} for an illustration of this notation.

\begin{figure}[t]
  \captionsetup{width=.8\linewidth}
  \includegraphics[width=6in]{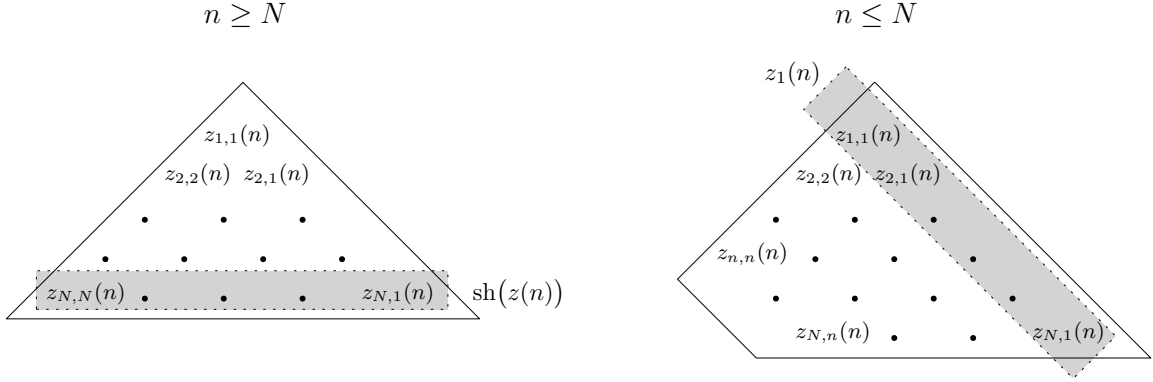}
  \caption{The array $z(n)$. When $n\geq N$  it is triangular and when $n\leq N$ it is trapezoidal. Also depicted on the left is the shape $\sh\big(z(n)\big)$ (i.e., bottom row), and on the right is the first diagonal $z_1(n)$.}
  \label{fig:arrays}
\end{figure}

%

We now describe how to construct the $P$ tableaux $z(n)$ from the matrix $m_{\llbracket 1,N\rrbracket,\llbracket 1,n\rrbracket}$ through insertion of the rows of $m_{\llbracket 1,N\rrbracket,\llbracket 1,n\rrbracket}$ into an empty tableaux. For each $n\in \mathbb{Z}_{\geq 1}$, let  $a_1(n) = m_{\llbracket 1,N\rrbracket,\llbracket n\rrbracket}$.  Then the mapping is given graphically on the left side of Figure \ref{fig:gRSKCOSZ} (the right side contains an example when $n=4$ and $N=3$). It is also possible to define this mapping through operators, though it is no more informative (and perhaps less so) than the graphical definition which we stick with.

\begin{figure}[h]
  \captionsetup{width=.8\linewidth}
  \includegraphics[width=6in]{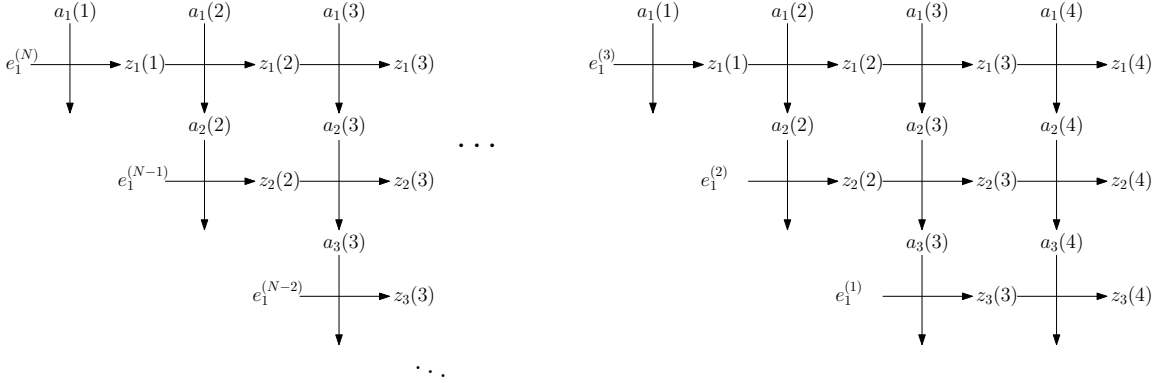}
  \caption{On the left: The matrix $m^{\llbracket 1,n\rrbracket}$ is inserted into an empty tableaux. The inputs on the top are specified by setting $a_1(n) =m_{\llbracket 1,N\rrbracket,\llbracket n\rrbracket}$. Everything else is computed sequentially (from top-left down and to the right) via the geometric row insertion. The output matrix $z(n)$ is read off from its diagonals $z_1(n),\ldots ,z_{n\wedge N}(n)$ on the right-side of the figure. On the right: the specific case when $n=4$ and $N=3$.}
  \label{fig:gRSKCOSZ}
\end{figure}

For later uses, let us call the above defined mapping $P$.
 The $Q$ tableaux is defined in terms of the sequence of shapes of $z(1),\ldots,z(n)$.
 Specifically,
\begin{equation}\label{eqPQdef}
P\big(m_{\llbracket 1,N\rrbracket,\llbracket 1,n\rrbracket}\big) :=z(n)\qquad
Q\big(m_{\llbracket 1,N\rrbracket,\llbracket 1,n\rrbracket}\big):= \big(\sh(z(t))\big)_{t\in \llbracket 1,n\rrbracket}.
\end{equation}
Figure \ref{fig:PQ} shows how the $P$ and $Q$ tableaux may be joined together to form another $n$ by $N$ matrix of strictly positive real numbers. Notice that the definition of the $Q$ tableaux implies that it is consistent as $n$ varies so that for $n'>n$, the first $n$ elements in the sequence of shapes that defines $Q\big(m_{\llbracket 1,N\rrbracket,\llbracket 1,n'\rrbracket}\big)$ equals $Q\big(m_{\llbracket 1,N\rrbracket,\llbracket 1,n\rrbracket}\big)$.

\begin{figure}[t]
  \captionsetup{width=.8\linewidth}
  \includegraphics[width=5.5in]{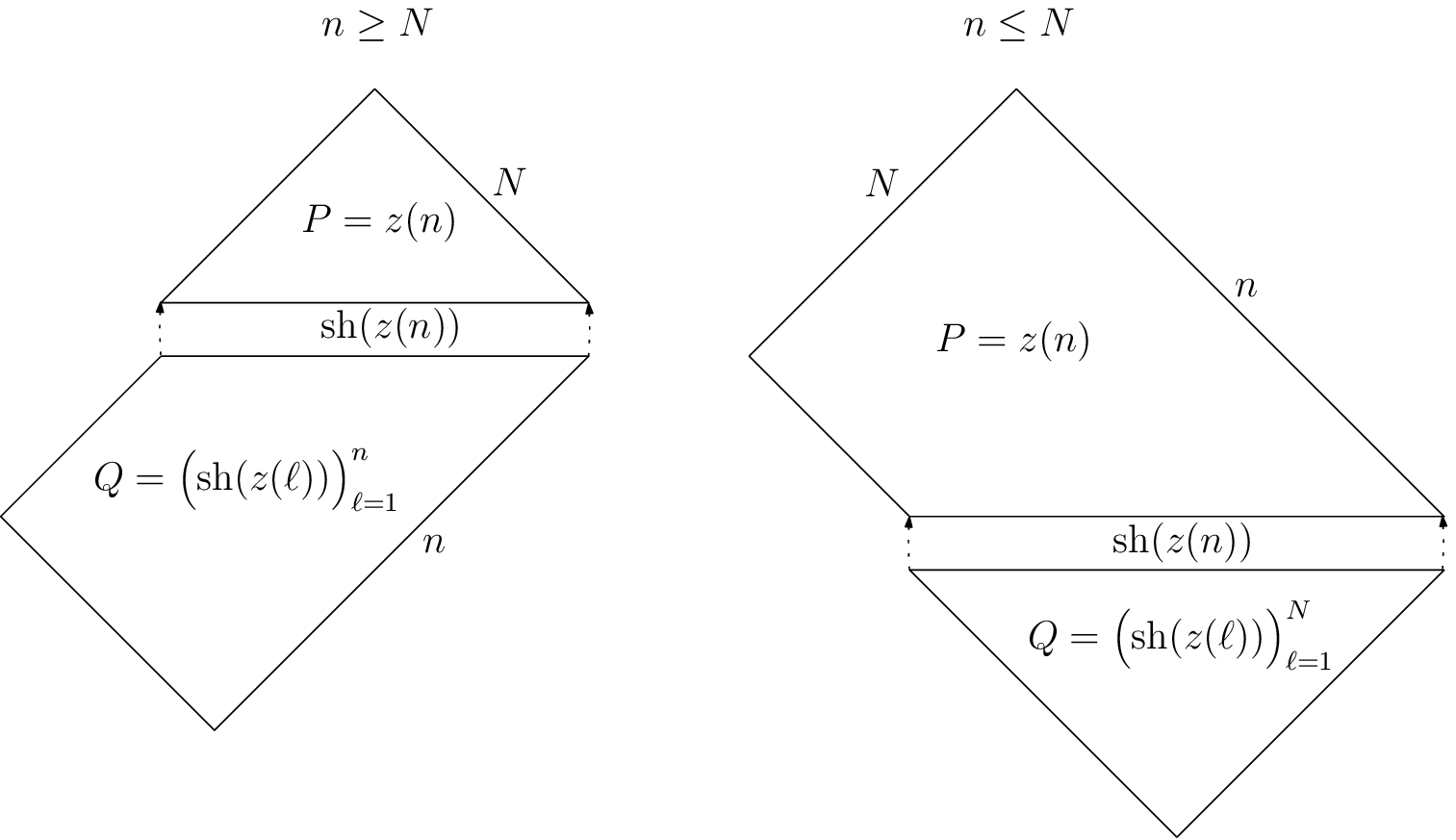}
  \caption{The $P$ and $Q$ tableaux agree in that the shape (bottom row) of the $P$ tableaux matches the top row of the $Q$ tableaux (where the shapes are listed from that of $z(n)$ down to $z(1)$). On the right this is shown when $n\geq N$ and on the right is the case $n\leq N$.}
  \label{fig:PQ}
\end{figure}
\end{defn}

Generalizing  Greene's theorem for the usual RSK correspondence, \cite{NY} showed that the geometric RSK also has a partition function interpretation. We will follow the exposition of \cite[Section 2.2]{COSZ} and define a mapping from the weight matrix $m_{\llbracket 1,N\rrbracket,\llbracket 1,n\rrbracket}$ to an array $\tilde{z}(n)$ of the same dimensions as $z(n)$ defined in \eqref{eq:zz}. The result quoted below in Proposition \ref{prop:rskequivs} shows that $z(n)$ and $\tilde{z}(n)$ are equal, and thus (in light of the way $\tilde{z}$ is defined) provides a  partition function  interpretation for $z(n)$.
\begin{defn}\label{def:path}
For $1\leq \ell\leq k\leq N$ let $\Pi^{\ell}_{n,k}$ denote the set of $\ell$-tuples $\pi=(\pi,\ldots, \pi_\ell)$ of non-intersecting lattice paths in $\ZZ^2$ with the property that for $1\leq r\leq \ell$, $\pi_r$ is a lattice path from $(1,r)$ to $(n,k+r-\ell)$. The term lattice path means that between nearest-neighbor lattice points in $\ZZ$, the path either takes a unit step up or right; the term non-intersecting means that the paths do not touch, even at lattice points. For any $\ell$-tuple of non-intersecting lattice paths $\pi = (\pi_1,\ldots, \pi_\ell)$, define their {\it weights} to be
$$
wt(\pi) = \prod_{r=1}^{\ell} \prod_{(i,j)\in \pi_r} m_{i,j}.
$$
For $1\leq \ell\leq k\leq N$ let
$$
\tau_{k,\ell}(n) = \sum_{\pi\in \Pi^{\ell}_{n,k}} wt(\pi).
$$
For $0\leq n<\ell<k\leq N$ the set $\Pi^{\ell}_{n,k}$ is empty and hence we set that $\tau_{k,\ell}(n)=0$ in that case. When $\ell=k$, there exists only a single $\pi\in \Pi^{\ell}_{n,k}$. In fact, whenever $0\leq n<\ell\leq k\leq N$, the set $\Pi^{\ell}_{n,k}$ is only non-empty in the case that $k=\ell$. In otherwords,
$$
\tau_{k,\ell}(n) = \delta_{k,\ell}\tau_{k,n}(n)\quad \textrm{whenever}\quad 0\leq n<\ell\leq k\leq N,
$$
where $\delta_{k,\ell}$ is the Kronecker delta function. Finally, for $\ell=0$ we adopt the convention that $\tau_{k,0} = 1$ for $1\leq k\leq N$.

Having defined the multi-path partition function $\tau_{k,\ell}(n)$, we can now  define the array $\tilde{z}(n)$ by the relation that for all indices $k$ and $\ell$,
$$
z_{k,1}(n)\cdots z_{k,\ell}(n) = \tau_{k,\ell}(n).
$$
\end{defn}

\begin{prop}[Proposition 2.5 of \cite{COSZ}]\label{prop:rskequivs}
The mapping $m_{\llbracket 1,N\rrbracket,\llbracket 1,n\rrbracket}\mapsto z(n)$ from Definitions \ref{NYdef} and $m_{\llbracket 1,N\rrbracket,\llbracket 1,n\rrbracket}\mapsto \tilde{z}(n)$  from Definition \ref{def:path} are the same, i.e., $z(n)=\tilde{z}(n)$.
\end{prop}
\begin{proof}
In \cite{COSZ}, the first of these mappings is denoted (with $m$ replaced therein by $d$) by $\emptyset\leftarrow m^{[1]}\leftarrow m^{[2]}\leftarrow\cdots \leftarrow m^{[n]}$ and the second is denoted by $P_{n,N}(m^{[1,n]})$. Proposition 2.5 of \cite{COSZ} then provides the equality, following methods used in \cite{NY}.
\end{proof}

\subsubsection{Rewriting the geometric RSK via the geometric Pitman transform}\label{sec:RSKPit}
We now explain how to the geometric row insertion and hence geometric RSK can be rewritten in terms of the Pitman transform introduced earlier in Definition \ref{defn:dpit}. For the usual RSK correspondence, this is explained, for instance, in notes of Pei available at \url{https://toywiki.xyz/}. In lifting this to the geometric setting, there are some subtleties which arise.

Let us recall the Pitman transform from Definition \ref{defn:dpit}. For any $r\in \mathbb{Z}$ and any functions $f,g:\mathbb{Z}_{\geq r}\to (0,\infty)$ define functions $\big(g\odot f\big):\mathbb{Z}_{\geq r}\to (0,\infty)$ and  $\big(f\otimes g\big):\mathbb{Z}_{\geq r+1}\to (0,\infty)$ by
\begin{align}
\big(g\odot f\big)(x)&:= f(x)\cdot \sum_{m=r}^{x} \frac{g(m)}{f(m-1)} &&\textrm{for }x\geq r\\
\big(f\otimes g\big)(x)&:= g(x)\cdot \Big(\sum_{m=r}^{x} \frac{g(m)}{f(m-1)}\Big)^{-1} &&\textrm{for }x\geq r+1\\
\end{align}
where we have adopted the convention that $f(r-1)=1$ in the denominator when $m=r$. The operator $\Tau$ defined in \eqref{eq:taudef} can be encoded graphically as
\begin{figure}[H]
  \includegraphics[width=1in]{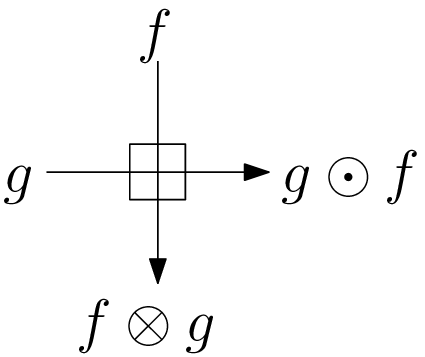}
\end{figure}
\noindent noting that the output on the right-side is still a function from $\mathbb{Z}_{\geq r}\to (0,\infty)$ whereas the bottom output is a function from $\mathbb{Z}_{\geq r+1}\to (0,\infty)$.

The next lemma rewrites the geometric row insertion in terms of the Pitman transform.
\begin{lemma}\label{lem:rskpit}
Let $\ell\in\llbracket 1,N\rrbracket$ and consider two words $\xi=(\xi_{\ell},\ldots,\xi_N)$ and $b=(b_\ell,\ldots, b_N)$ with strictly positive real entries. Compute $\xi'$ and $b'$ (where $\xi'=(\xi'_{\ell},\ldots, \xi'_N)$ and $b'=(b'_{\ell+1},\ldots, b'_N)$) via the geometric row insertion of $b$ into $\xi$ as depicted in the left-hand side in Figure \ref{fig:RSKPit}. Now define functions $\xi,B:\mathbb{Z}_{\geq \ell}\to (0,\infty)$ such that $\xi(k)=\xi_k$ for all $k\in \llbracket \ell,N\rrbracket$ and such that for all $k\in \llbracket \ell,N\rrbracket$ we have  $B(k) = \prod_{j=\ell}^k b_j$. The values of the functions $\xi$ and $B$ for arguments larger than $N$ do not matter and can be set to 1 for concreteness. Now, compute $\xi'$ and $B'$ via the Pitman transform as depicted in the right-hand side in Figure \ref{fig:RSKPit}. Then $\xi'(k)=\xi'_k$ for all $k\in \llbracket \ell,N\rrbracket$ and $B'(k) = \prod_{j=\ell+1}^k b'_j$ for all $k\in \llbracket \ell+1,N\rrbracket$. In other words, the geometric row insertion and Pitman transform produce the same result.
\begin{figure}[ht]
  \captionsetup{width=.8\linewidth}
  \includegraphics[width=2in]{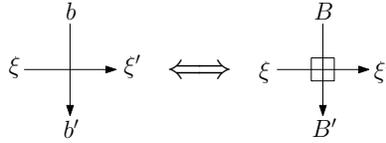}
  \caption{The equivalence between the geometric row insertion and the Pitman transform as shown in Lemma \ref{lem:rskpit}.}
  \label{fig:RSKPit}
\end{figure}

\end{lemma}
\begin{proof}
Owing to \eqref{g-row-ins} we can compute explicitly $\xi'_k$ for all $k\in \llbracket \ell,N\rrbracket$:
\begin{equation}\label{eq:xiprimek}
\xi'_k = \sum_{j=\ell}^{k} b_j\cdots b_k \xi_j = B(k) \sum_{j=\ell}^{k}\frac{\xi(j)}{B(j-1)}= \big(\xi\odot B\big)(k)=\xi'(k).
\end{equation}
The first equality is easily shown by induction in $k$, the second is verified by substituting the definitions of the functions $B$ and $\xi$, the third equality is by definition of $\odot$ and the final equality is the definition of the function $\xi'$.

Again appealing to  \eqref{g-row-ins}, we know that $b'_j = b_j \frac{\xi_j\xi'_{j-1}}{\xi_{j-1}\xi'_j}$ for all $j\in \llbracket \ell+1,N\rrbracket$. Taking the product of this equality over $j\in \llbracket \ell+1,k\rrbracket$, we see that
$$
 \prod_{j=\ell+1}^k b'_j = \frac{B(k)}{B(\ell)} \frac{\xi(k)\xi'(\ell)}{\xi(\ell)\xi'(k)} = \frac{B(k)\xi(k)}{\xi'(k)} = B\otimes \xi(k)=B'(k).
$$
The first equality comes from the explicit formula for $b'_j$ recalled above and the definition of the functions $B, \xi$ and $\xi'$, the second equality comes cancelations by substituting the equality $\xi'(\ell)=B(\ell)\xi(\ell)$, the third equality uses the formula for
$\xi'(k)$ from \eqref{eq:xiprimek} (after the second equality sign), the fourth equality is by definition of $\otimes$ and the final equality is the definition of the function $B'$.
\end{proof}

From Lemma \ref{lem:rskpit}, we rewrite the entire gRSK correspondence via the Pitman transform.

\begin{figure}[t]
  \captionsetup{width=.8\linewidth}
  \includegraphics[width=5in]{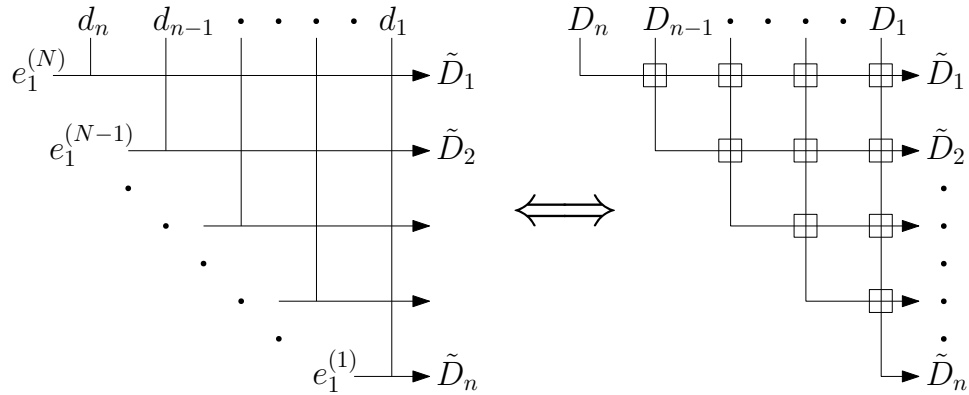}
  \caption{The equivalence between the gRSK correspondence (on the left) and Pitman transform (on the right) as shown in Lemma \ref{lem:fullrskmatch}. We have dropped all internal labels. The outgoing lines (coming out to the right and below each vertex) carries the output from each vertex to the next vertices. In the depiction of the gRSK correspondence, we have dropped the arrows coming out below the vertices in which the $e_1$ vectors are inserted. This is justified since that output is empty anyway.}
  \label{fig:fullRSKmatch}
\end{figure}

\begin{lemma}\label{lem:fullrskmatch}
Fix any $n,N\in \mathbb{Z}_{\geq 1}$ and vectors $d_1,\ldots, d_n\in (0,\infty)^N$ so that $d_i=\big(d_{i,1},\ldots, d_{i,N}\big)$. Compute vectors $\tilde{D}_1,\ldots,\tilde{D}_n$ via the gRSK correspondence as depicted in the left-hand side of Figure \ref{fig:fullRSKmatch}. Define functions $D_1,\ldots, D_n:\mathbb{Z}_{\geq 1}\to (0,\infty)$ such that for each $i\in \llbracket 1,n\rrbracket$ and for each $k\in \llbracket 1,N\rrbracket$ we have $D_i(k) = \prod_{j=1}^k d_{i,j}$. For $k>N$ the value of the $D_i$ does not matter and can be set to 1 for concreteness. Compute $n$ functions $\tilde{D}_1,\ldots, \tilde{D}_n$ where $\tilde{D}_i:\mathbb{Z}_{\geq i}\to (0,\infty)$ for $i\in \llbracket 1,n\rrbracket$ via the Pitman transform as depicted in the right-hand side of Figure \ref{fig:fullRSKmatch}. The, the outcome of these two calculations match in the sense that $\tilde{D}_{i,j} = \tilde{D}_{i}(j)$ as long as $1\leq i\leq j\leq N$ and $i\in \llbracket 1,n\wedge N\rrbracket$.
\end{lemma}
\begin{proof}
The follows immediately from Lemma \ref{lem:rskpit} along with the insertion rule \eqref{es-row-ins} used on the boundary in the gRSK correspondence.
\end{proof}

The Pitman transform depicted on the right-hand side of \eqref{fig:fullRSKmatch} is a graphical implementation of the $\mathcal{W}$ operator from \eqref{eq:SW}. Thus, in light of Lemma \ref{lem:fullrskmatch}, we have the following.

\begin{cor}\label{cor:DWD}
Let $\tilde{D}=(\tilde{D}_1,\ldots, \tilde{D}_n)$ denote the output on the right-hand side of \eqref{fig:fullRSKmatch} with given input functions $D_1,\ldots, D_n$. Then
$$
\tilde{D} = \mathcal{W}D.
$$
\end{cor}

Combining Corollary \ref{cor:DWD} and Proposition \ref{prop:rskequivs} yields a partition function formula for $\mathcal{W}D$.

\begin{figure}[ht]
  \captionsetup{width=.8\linewidth}
  \includegraphics[width=2.5in]{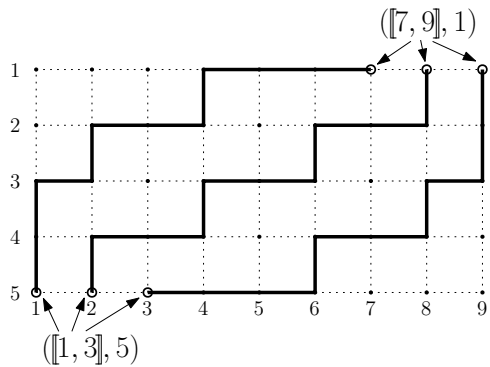}
  \caption{Corollary \ref{cor:Wpart} provides a partition function representation for $\mathcal{W}D$. Illustrated here is the endpoint pair (and a multipath between them) associated to $n=5$, $N=9$ and $\ell=3$.}
  \label{fig:Wpart}
\end{figure}
\begin{cor}\label{cor:Wpart}
Fix any $n, N\in \mathbb{Z}_{\geq 1}$ and any functions  $D_1,\ldots, D_n:\mathbb{Z}_{\geq 1}\to (0,\infty)$, writing $D=(D_1,\ldots, D_n)$. Then, for $\ell\in \llbracket 1,n\wedge N\rrbracket$,
$$
\prod_{r=1}^\ell \big(\mathcal{W}D\big)_r(N) = D\Big[ \big(\llbracket 1,\ell\rrbracket,n\big)\to \big(\llbracket N-\ell+1,N\rrbracket,1\big)\Big],
$$
where we have used the shorthand that for $a\leq b\in \mathbb{Z}$,
$\big(\llbracket a,b\rrbracket,m\big) := \big((a,m),\ldots, (b,m)\big)$ (see Figure \ref{fig:Wpart}).
Letting $d_{i,j}:=\frac{D_i(j)}{D_i(j-1)}$ (with $D_i(0)\equiv 1$) and defining the $N$ by $n$ array $\widetilde{d}$ via $\widetilde{d}_{i,j} = d_{i,n-j+1}$,
$$
\big(\mathcal{W}D\big)(N) = \sh\big(P(\widetilde d_{\llbracket 1,N\rrbracket,\llbracket 1,n\rrbracket})\big).
$$
Consequently, letting $\tilde{d}^\top$ be the transpose of $\widetilde{d}$ so $\widetilde{d}^\top_{i,j} = \widetilde{d}_{j,i}$, we have that for any $N$
$$
\Big(\big(\mathcal{W}D\big)(t)\Big)_{t\in \llbracket 1,N\rrbracket} = Q\big(\widetilde{d}^\top_{\llbracket 1,n\rrbracket,\llbracket 1,N\rrbracket})\big).
$$
\end{cor}

\subsubsection{Inverse-gamma weights}\label{sec:inv}
A random variable $X$ has {\it inverse-gamma distribution with parameter $\theta>0$} if it is supposed on the positive reals where it has density relative to Lebesgue given by
$\frac{1}{\Gamma(\theta)} x^{-\theta-1}e^{-\frac{1}{x}}dx.$

If the array $m$ from Definition \eqref{NYdef} is filled with iid inverse-gamma distributed random variables, \cite[Theorem 3.9]{COSZ} shows that the shape $\sh(z(n))$ of $m_{\llbracket 1,N\rrbracket,\llbracket 1,n\rrbracket}$ under the gRSK correspondence of evolves as a Markov process in $n$ with an explicit transition kernel (between $n$ and $n+1$). It follows directly from the factorized nature of the transition kernel that  $\big(\sh(z(n))\big)_{n\in\mathbb{Z}_{\geq 1}}$ enjoys the structure of a (discrete) Gibbsian line ensemble -- see \cite{Wu,JOC,BCD} for details. In Section \ref{sec:airykpzsheet} we mention how a special limit of this Gibbsian structure is important in the Airy sheet construction of \cite{DOV} from the Airy line ensemble.

\subsubsection{Proof of Theorem \ref{thm:dinv} via Desnanot-Jacobi}\label{sec:dj}
This proof was communicated to us by Konstantin Matveev. It has the flavor of methods used to parameterize totally positive matrices (see, e.g. \cite{FZ} and references therein), though the matrices presently are totally non-negative, for which the parameterization is more involved.

Recall that we seek to prove that $D\big[U\to V\big] = \big(\mathcal{W}D\big)\big[\ua U\to V\big]$. We will first relate this equality, by means of the Lindstr\"{o}m-Gessel-Viennot lemma, to an equality of matrices $M$ and $\widetilde{M}$ whose entries are the $k=1$ versions of this identity (when $|U|=|V|=1$). The idea to proving the equality of those matrices is to match certain minor determinants and then use the Desnanot-Jacobi identity to inductively match all of the others.

Before commencing with the proof, let us recall the Desnanot-Jacobi identity. For a matrix $M$, we write $M_{U,V}$ for the minor of $M$ comprised of rows $U$ and columns $V$. For instance, $M_{\llbracket i,i+k-1\rrbracket,\llbracket j,j+k-1\rrbracket}$ is the minor composed of rows $i$ through $i+k-1$, and columns $j$ through $j+k-1$. Writing $|M|$ for the determinant $M$, the Desnanot-Jacobi identity states that
\begin{align}
\nonumber\big|M_{\llbracket i,i+k-1\rrbracket,\llbracket j,j+k-1\rrbracket}\big|\cdot \big|M_{\llbracket i-1,i+k-2\rrbracket,\llbracket j-1,j+k-2\rrbracket}\big| &= \big|M_{\llbracket i-1,i+k-1\rrbracket,\llbracket j-1,j+k-1\rrbracket}\big|\cdot \big|M_{\llbracket i,i+k-2\rrbracket,\llbracket j,j+k-2\rrbracket}\big|\\ &+ \big|M_{\llbracket i-1,i+k-2\rrbracket,\llbracket j,j+k-1\rrbracket}\big|\cdot \big|M_{\llbracket i,i+k-1\rrbracket,\llbracket j-1,j+k-2\rrbracket}\big|.
\label{eq:dj}
\end{align}

To ease notations, let us define $\widetilde{D} = \mathcal{W}D$ and from $D$ and $\widetilde{D}$ we define vertex weights $d_{x,m}=D_m(x)/D_m(x-1)$ and $\widetilde{d}_{x,m}=\widetilde{D}_m(x)/\widetilde{D}_m(x-1)$. Note that while $d_{x,m}$ is defined for all $x\in \mathbb{Z}_{\geq 1}$ and $m\in \llbracket 1,n\rrbracket$,  $\widetilde{d}_{x,m}$ is only defined when we additionally assume that $x\geq m$. Let us also define matrices $M$ and $\widetilde{M}$ by (see \ref{fig:Dmat} for an illustration)
$$
M_{u,v} :=D\big[(u,n)\to (v,1)\big]\qquad \widetilde{M}_{u,v}:=\widetilde{D}\big[(u,n\wedge u)\to (v,1)\big].
$$
Our first observation is that by applying the  Lindstr\"{o}m-Gessel-Viennot (LGV) lemma, it follows that
$$
D\big[U\to V\big] =\big|M_{U,V}\big|,\qquad \textrm{and}\qquad \widetilde{D}\big[\ua U\to V\big] = \big|\widetilde{M}_{U,V}\big|.
$$
We have slightly abused notation above. On the left-hand sides of the above equations, we take $U=\big((u_1,n),\cdots ,(u_k,n)\big)$ and $V=\big((v_1,1),\ldots, (v_k,1)\big)$ with $u_1<\cdots <u_k$ and $v_1<\cdots <v_k$; on the right-hand sides we take $U=(u_1,\ldots, u_k)$ and $V=(v_1,\ldots, v_k)$.
where $u,v\in \mathbb{Z}_{\geq 1}$. Note that the LGV lemma is typically stated for edge-weighted graphs, not vertex-weighted. However, in our present case we can transfer our vertex weights to edge weights in the following way. To every edge entering a vertex from the left or below, associate the edge with the weight of the vertex. For vertices on the bottom or left-diagonal of the diagrams in Figure \ref{fig:Dmat}, attach an extra vertical edge below them and transfer the vertex weight to that edge. Then, the edge-weighted partition function from $u$ (where the starting vertex for $u$ is now shifted down by one to be below the newly added edge) to $v$ is equal to the original vertex-weighted partition function. The same holds for the partition functions for ensembles of non-intersecting multipaths from $U$ to $V$, hence we can apply LGV.

\begin{figure}[t]
  \captionsetup{width=.8\linewidth}
  \includegraphics[width=4in]{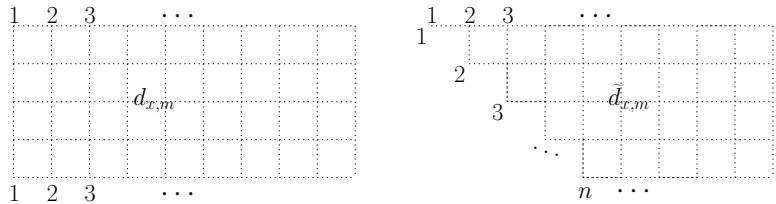}
  \caption{On the left we depict the matrix $M$ and on the right the matrix $\widetilde{M}$. The rows are indexed by the numbers on the bottom (and left-diagonal in the case of $\widetilde{M}$) and the columns are indexed by the numbers on the top of the diagrams. The matrix elements with row $u$ and column $v$ are equal to the sum over all directed paths from $u$ to $v$ of the product of the vertex weights over the paths. On the left, the vertex weights are given by the $d_{x,m}$ and on the right by the $\widetilde{d}_{x,m}$.}
  \label{fig:Dmat}
\end{figure}

In light of the above observation and application of the LGV lemma, it suffices to prove equality of the matrix elements of $M$ and $\widetilde{M}$. To prove this we will appeal to certain a priori equalities of minor determinants of $M$ and $\widetilde{M}$, along with the Desnanot-Jacobi identity. Let us start by recording the a priori obvious equalities:
\begin{enumerate}
\item For all $j,k\in \mathbb{Z}_{\geq 1}$, $\big|M_{\llbracket 1,k\rrbracket,\llbracket j,j+k-1\rrbracket}\big| = \big|\widetilde{M}_{\llbracket 1,k\rrbracket,\llbracket j,j+k-1\rrbracket}\big|$. This is from the definition of $\widetilde{D}$.
\item $\big|M_{\llbracket i,i+k-1\rrbracket,\llbracket j,j+k-1\rrbracket}\big|$ and $\big|\widetilde{M}_{\llbracket i,i+k-1\rrbracket,\llbracket j,j+k-1\rrbracket}\big|$ are strictly positive if and only if $k\leq n$ and $i\leq j$, or $k>n$ and $i=j$. In the other cases, the minor determinants are zero. This follows because these determinants present partition functions which involve positive weights. In the case when $k>n$, the only non-zero partition function is when $i=j$.
\item $\big|M_{\llbracket i,i+k-1\rrbracket,\llbracket i,i+k-1\rrbracket}\big|= \big|\widetilde{M}_{\llbracket i,i+k-1\rrbracket,\llbracket i,i+k-1\rrbracket}\big|$. This follows from the calculation depicted in Figure \ref{fig:ratio} and is a consequence of the upper-triangularity of the matrices in question.
\item $\big|M_{\llbracket i,i+k-1\rrbracket,\llbracket j,j+k-1\rrbracket}\big|= \big|\widetilde{M}_{\llbracket i,i+k-1\rrbracket,\llbracket j,j+k-1\rrbracket}\big|=0$ if $i<j$ and $k>n$. This is because it is not possible to connect off-set starting and ending points by more than $n$ disjoint paths.
\end{enumerate}

\begin{figure}[t]
  \captionsetup{width=.8\linewidth}
  \includegraphics[width=6in]{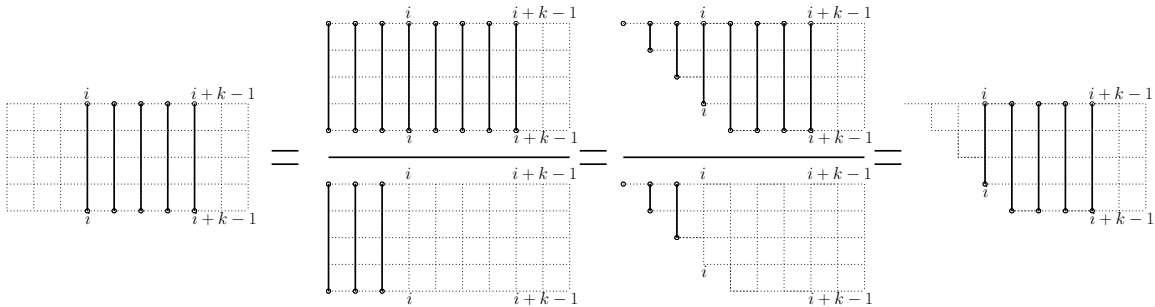}
  \caption{On the left we represent (in light of the LGV lemma) the partition function equal to $\big|M_{\llbracket i,i+k-1\rrbracket,\llbracket i,i+k-1\rrbracket}\big|$. There is only one set of paths (as depicted in the figure) which contributes and the weights used are those of $d_{x,m}$. The first equality rewrites the contribution of that path as a ratio of two partition functions (still using the $d_{x,m}$ weights) which start at the left edge of the diagram. Both of these also only have single paths which contribute. The second equality follows by observing that both the numerator and denominator are equal to the corresponding numerator and denominator for the diagram with the corner removed and with weights $\widetilde{d}_{x,m}$. This equality was observe as the first a prior equality and is from the definition of $\widetilde{D}$. The final equality is again reliant on the fact that only one path contributes to the partition functions.}
  \label{fig:ratio}
\end{figure}

We will now inductively show that for all $i,j,k\in \mathbb{Z}_{\geq 1}$
\begin{equation}
\big|M_{\llbracket i,i+k-1\rrbracket,\llbracket j,j+k-1\rrbracket}\big|=\big|\widetilde{M}_{\llbracket i,i+k-1\rrbracket,\llbracket j,j+k-1\rrbracket}\big|.
\label{eq:toshowMMtilde}
\end{equation}
We prove this by induction. As a base case, from the a priori equalities listed earlier we know that this is true for $i=j$ and all $k$, as well as for $i=1$ and all $j\geq i$ and $k$. The induction is on $(i,j+k-1)$ in lexicographic order and relies on the Desnanot-Jacobi identity \eqref{eq:dj}. As long as
$\big|M_{\llbracket i-1,i+k-2\rrbracket,\llbracket j-1,j+k-2\rrbracket}\big|>0$, we may divide both sides of \eqref{eq:dj} by it, yielding
\begin{align}
{\scriptstyle|M_{\llbracket i,i+k-1\rrbracket,\llbracket j,j+k-1\rrbracket}|}= \tfrac{|M_{\llbracket i-1,i+k-1\rrbracket,\llbracket j-1,j+k-1\rrbracket}|\cdot |M_{\llbracket i,i+k-2\rrbracket,\llbracket j,j+k-2\rrbracket}|+ |M_{\llbracket i-1,i+k-2\rrbracket,\llbracket j,j+k-1\rrbracket}|\cdot |M_{\llbracket i,i+k-1\rrbracket,\llbracket j-1,j+k-2\rrbracket}|}{|M_{\llbracket i-1,i+k-2\rrbracket,\llbracket j-1,j+k-2\rrbracket}| }
\label{eq:djprime}
\end{align}
which can be graphically depicted in Figure \ref{fig:dj}.

\begin{figure}[t]
  \captionsetup{width=.8\linewidth}
  \includegraphics[width=6in]{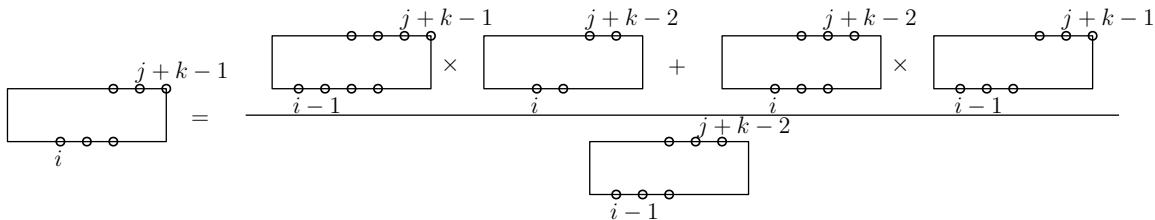}
  \caption{A graphical depiction of the identity (\ref{eq:djprime}) which is essentially the Desnanot-Jacobi identity (\ref{eq:dj}) up to dividing by one of the terms on the left. The diagrams represent minor determinants of the matrix $M$. If the bottom label is $i$ and the top label is $j+k-1$ and the number of circles $\circ$ is $j$, then the diagram represents the determinant $|M_{\llbracket i,i+k-1\rrbracket,\llbracket j,j+k-1\rrbracket}|$ (as on the left-hand side of the equation). Notice that the number of circles and the starting and ending points differ in each term on the right-hand side, but in manner lexicographically dominated by $(i,j+k-1)$.}
  \label{fig:dj}
\end{figure}

On the left-hand side of \eqref{eq:djprime} we are dealing with paths which start left-justified at $i$ and ending right-justified at $j+k-1$. On the right-hand side of \eqref{eq:djprime} we either have the starting point $i$ is diminished to $i-1$, or that it remains $i$ but the ending point is diminished to $j+k-2$. Note that the number of paths may change, but we are more interested in keeping track of the left starting point and right ending point since that is what we induct upon. Inductively from the based cases, we establish the positivity of the denominator and that we may replace the $M$ by $\widetilde{M}$ on the right-hand side of \eqref{eq:djprime}. Indeed, the base case is sufficient for this induction since repeated application of \eqref{eq:djprime} eventually yields expressions of the form $|M_{\llbracket \tilde i,\tilde i+\tilde k-1\rrbracket,\llbracket \tilde j,\tilde j+\tilde k-1\rrbracket}|$ with either $\tilde i=1$ or $\tilde i=\tilde j$.

To summarize, via the above described induction, we show the equality desired in \eqref{eq:toshowMMtilde}, precisely as needed to prove the theorem.

\subsubsection{Proof of Theorem \ref{thm:dinv} via the Noumi-Yamada matrix encoding of gRSK}\label{sec:ny}
This proof essentially comes from combining the results of Theorem 1.7 and Section 2.3 of \cite{NY}. As indicated in the introduction, we appreciate Duncan Dauvergne pointing us towards this. For completeness, we describe the proof in the notation of our present paper.

From a function $E:\mathbb{Z}_{\geq r} \to (0,\infty)$ (with the convention that $E(r-1)\equiv 1$) define $e_x= E(x)/E(x-1)$ for $x\in \mathbb{Z}_{\geq r}$. From this we define a matrix $H_r(E):\mathbb{Z}_{\geq 1}^2\to [0,\infty)$ whose $i,j$ entry is equal to the partition function from $i$ on the bottom to $j$ on the top of the diagram depicted in Figure \ref{fig:hs}. More concretely, for $i=j\in \llbracket 1,r-1\rrbracket$ the entries are $1$ and for $i,j\in \mathbb{Z}_{\geq r}$ with $i\leq j$, the entry is the partial product $e_i\cdots e_j$; all other entries are zero.

\begin{figure}[H]
  \captionsetup{width=.8\linewidth}
  \includegraphics[width=4in]{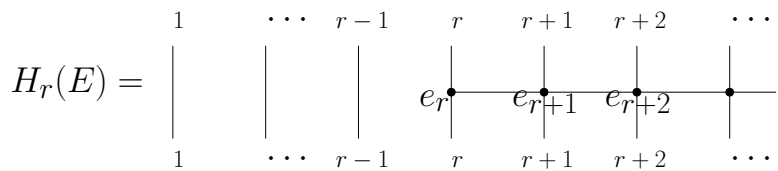}
  \caption{The matrix entry $(i,j)$ of $H_r(E)$ corresponds to the partition function from $i$ on the bottom to $j$ on the top. The weights of the edges are all 1 and the weights of the vertices are as labeled.}
  \label{fig:hs}
\end{figure}

The starting point for this proof is the observation made in Section 2.2 of \cite{NY} that geometric row insertion can be related to matrix factorization via these  $H_r$ matrices. For our purposes, this observation boils down to the following fact: For any $r\in \mathbb{Z}_{\geq 1}$ and $G,F:\mathbb{Z}_{\geq r}\to (0,\infty)$, the geometric Pitman transforms $\big(g\odot f\big):\mathbb{Z}_{\geq r}\to (0,\infty)$ and  $\big(f\otimes g\big):\mathbb{Z}_{\geq r+1}\to (0,\infty)$ defined in Definition \ref{defn:dpit} satisfy the matrix identity
\begin{equation}
H_r(G)H_r(F) = H_{r+1}(F\otimes G)H_r(G\odot F).
\label{eq:nyH}
\end{equation}
This identity is readily checked from definitions, and in fact is essentially the same as the result from Step 1 of the proof of Theorem \ref{thm:dinv} given in Section \ref{sec:dpart}.

As shown in Corollary \ref{cor:DWD}, the gRSK correspondence can be written in terms of the Pitman transforms (see also the right-hand side of Figure \ref{fig:fullRSKmatch}). Given this and \eqref{eq:nyH} it follows immediately that $D$ and  $\widetilde{D} := \mathcal{W}D$ satisfy the matrix identity
\begin{equation}
H_1(D_n)\cdots H_1(D_1) = H_{n}(\widetilde{D}_n)\cdots H_{1}(\widetilde{D}_1).
\label{eq:nyHfull}
\end{equation}

Multiplying the $H$ matrices is equivalent to appending the diagrams and computing matrix entries as partition functions. From this, it follows that the $(u,v)$ entry corresponding to the left-hand side of \eqref{eq:nyHfull} are precisely $D\big[(u,n)\to (v,1)\big]$ and the $(u,v)$ entry corresponding to the right-hand side of \eqref{eq:nyHfull} are precisely $\widetilde{D}\big[(u,n\wedge u)\to (v,1)\big]$. The general $U$ and $V$ result follows (as in the Matveev proof) by an application of the Lindstr\"{o}m-Gessel-Viennot lemma.

\subsection{Zero-temperature limit}\label{sec:dzero}

All of the results in Sections \ref{sec:dpart} and  \ref{sec:dRSK} admit zero-temperature limits. We focus here only on the limit of Theorem \ref{thm:dinv}. For an {\it inverse-temperature} $\beta\in(0,\infty)$ and $f_1,\ldots, f_n:\mathbb{Z}_{\geq 1}\to \mathbb{R}$ (with the convention that $f_i(0)\equiv 0$) define $D^{\beta}_1,\ldots, D^{\beta}_n:\mathbb{Z}_{\geq 1}\to (0,\infty)$ via $D_i(x):= e^{\beta f_i(x)}$. For an endpoint pair $(U,V)$, define (note that the $\beta^{-1}$ superscript below is not an exponent, but rather a label for the variant of $f\big[U\to V\big]$ scaled as below)
$$
f\big[U\to V\big]^{\beta^{-1}} := \beta^{-1} \log\Big( D^{\beta}\big[U\to V\big]\Big).
$$
For any set of points  in the domain of $D$, observe that
$$f\big[U\big]^{\beta^{-1}} :=  \beta^{-1} \log\Big( D^{\beta}\big[U\big]\Big) = \sum_{(x,m)\in U} \big(f_m(x)-f_m(x-1)\big).
$$
Thus, by Laplace's method we can extract a {\it zero-temperature limit}
$$
\lim_{\beta\to \infty} f\big[U\to V\big]^{\beta^{-1}} = \max_{\pi\in U\to V} \Big(\sum_{(x,m)\in \pi} \big(f_m(x)-f_m(x-1)\big)\Big)=: f\big[U\to V\big]^0.
$$
We can also define zero-temperature limits of the operator $\Tau$ and $\mathcal{W}$ from Definition \ref{defn:dpit}. In particular, for any two functions $f,g:\mathbb{Z}_{\geq r}\to \mathbb{R}$, we define the zero-temperature Pitman transform via (we put a superscript $0$ to denote that this is a zero-temperature version)
\begin{align}
\big(g\odot f\big)^0(x)&:= f(x) + \max_{m\in\llbracket r,x\rrbracket} \big(g(m)-f(m-1)\big) &&\textrm{for }x\geq r\\
\big(f\otimes g\big)^0(x)&:= g(x) - \max_{m\in\llbracket r,x\rrbracket} \big(g(m)-f(m-1)\big) &&\textrm{for }x\geq r+1\\
\end{align}
(where by convention $f(r-1)=0$ when $m=r$). Then, the operator $\Tau^0$ is defined as
\begin{equation}\label{eq:taudefd}
\Tau^0\big(f,g\big) := \Big(\big(g\odot f\big)^0,\big(f\otimes g\big)^0\Big).
\end{equation}
As in Definition \ref{defn:dpit} we define $\Tau_{r,m}^0$ and then define
\begin{equation}\label{eq:SWd}
\mathcal{S}^0_rf:=\Tau_{r,r}^0\Tau_{r,r+1}^0\cdots \Tau_{r,n-2}^0\Tau_{r,n-1}^0f,\qquad
\mathcal{W}^0f :=\mathcal{S}^0_{n-1}\mathcal{S}^0_{n-2}\cdots \mathcal{S}^0_{2} \mathcal{S}^0_{1}f.
\end{equation}
Recalling the notation $\ua U$ from the statement of Theorem \ref{thm:dinv}, we may now state the immediate zero-temperature limit of that result.
\begin{cor}\label{cor:dzero}
Let $U=\big((u_i,n)\big)_{i\in\llbracket 1,k\rrbracket}$ and $V=\big((v_i,1)\big)_{i\in\llbracket 1,k\rrbracket}$ be any endpoint pair. Then, for any $n$ functions $f_1,\ldots, f_n:\mathbb{Z}_{\geq 1}\to \mathbb{R}$, writing $f=(f_1,\ldots, f_n)$ we have
\begin{equation}\label{eq:dinvd}
f\big[U\to V\big]^0 = \big(\mathcal{W}^0f\big)\big[\ua U\to V\big]^0.
\end{equation}
\end{cor}

The zero-temperature limit of the gRSK correspondence and its relation to the Pitman transform is essentially contained in the earlier mentioned online notes of Pei available at \url{https://toywiki.xyz/}. There are analogs of the results mentioned in Section \ref{sec:inv} when the inverse-gamma random variables are replaced by geometric or exponential random variables -- see the discussion after \cite[Proposition 4.1]{COSZ}.

\section{Semi-discrete polymers}\label{sec:sd}
The results contained in this section could be derived as limits of the discrete results from Section \ref{sec:d}. However, as the proofs are considerably simpler in this semi-discrete (semi-continuous) setting, we will provide complete and direct proofs. In comparing the results and proofs below to those in Section \ref{sec:d} a reader will also notice that we will now be working with logarithmic variables (i.e., the logarithms of the limits of variables from the discrete setting). We do this so as to provide a direct comparison with the zero temperature work of \cite{DOV}, as well as to directly connect our work to that of \cite{ON12}.

Before stating our semi-discrete result, Theorem \ref{thm:sdinv}, we first introduce the semi-discrete version of non-intersecting (multi)paths, partition functions and the Pitman transform.

\subsection{Partition function invariance}\label{sec:sdpart}

\begin{defn}[Non-intersecting (multi)paths]
  Fix $n\in \mathbb{Z}_{\geq 2}$. We will consider paths in the semi-discrete space $\mathbb{R}\times \llbracket 1,n\rrbracket$. For any $u<v\in \mathbb{R}$ and $m\leq \ell\in \llbracket 1,n\rrbracket$, we call a function $\pi:[u,v]\to \llbracket m,\ell\rrbracket$ a {\it path} with {\it starting points} $(u,m)$ and {\it ending points} $(v,\ell)$ if $\pi$ is nonincreasing, cadlag on $(u,v)$ and satisfies $\pi(u)=\ell$ and $\pi(v)=m$. We will often be interest in multiple paths $\pi_1,\ldots,\pi_k$ with respective starting and endpoint points
  \begin{equation*}
  U= \big\{(u_i,\ell_i)\big\}_{i\in \llbracket 1,k\rrbracket}\qquad\textrm{and}\qquad V= \big\{(v_i,\ell_i)\big\}_{i\in \llbracket 1,k\rrbracket},
  \end{equation*}
   in which case we will write $\pi=(\pi_1,\ldots, \pi_k)$ and call $\pi$ a {\it multipath} from $U$ to $V$. Two paths $\pi_1$ and $\pi_2$ are called {\it non-intersecting} if for all $t\in (u_1,v_2)\cap (u_2,v_2)$, $\pi_1(t)<\pi_2(t)$. This condition enforces that paths are disjoint in the interior of their common domain of definition. A multipath $\pi$ is non-intersecting if for each $1\leq i\neq j\leq k$, $\pi_i$ and $\pi_j$ are non-intersecting. We will assume that all multipaths are non-intersecting. We denote the set of (non-intersecting) multipaths $\pi$ from $U$ to $V$ by $U\to V$. If the set $U\to V$ is non-empty, then we say that the pair $(U,V)$ constitute an {\it endpoint pair}.
\end{defn}

\begin{defn}[Partition functions]\label{def:sdpartf}
Fix $n$ continuous functions $f_1,\ldots, f_n:\mathbb{R}_{\geq 0}\to \mathbb{R}$ centered so that $f_i(0)=0$ and write $f=(f_1,\ldots, f_n)$. To a single path $\pi$ from $(u,\ell)$ to $(v,m)$ we associated a {\it energy} to $\pi$ with respect to $f$ given by
\begin{equation*}
\int df\circ \pi :=\int_u^v f'_{\pi(t)}(t)dt.
\end{equation*}
Equivalently, to any path $\pi$ we can associate {\it jump times} $u=t_{\ell}<\ldots< t_{m}<t_{m-1}=v$ so that for all $j\in \llbracket m,\ell\rrbracket$, $t_j$ is the first time that $\pi$ is at level $j$. Then
\begin{equation*}
\int df\circ \pi = \sum_{j=m}^{\ell} f_j(t_{j-1})-f_{j}(t_{j}).
\end{equation*}
To a (multi)path $\pi=(\pi_1,\ldots, \pi_k)$ we associated an energy
\begin{equation*}
\int df\circ \pi  := \sum_{i=1}^k \int df\circ \pi_i.
\end{equation*}
For any endpoint pair $(U,V)$ we associate the {\it free energy}
\begin{equation}\label{eq:freeenergy}
f\big[U\to V\big] := \log\Big(\int_{\pi\in U\to V} e^{\int df\circ \pi}\Big)
\end{equation}
whose exponential is called the {\it partition function}. In \eqref{eq:freeenergy}, the integral over the set of $\pi\in U\to V$ should be understood as Lebesgue integral over the simplex of all possible jump times which define multipaths connecting $U$ to $V$.
\end{defn}

\begin{defn}[Semi-discrete geometric Pitman transform]\label{def:sdpit}
Define $n-1$ operators $\Tau_1,\ldots, \Tau_n$ which act on $n$-tuples of functions $f_1,\ldots,f_n$ (which we write as $f(t) = \big(f_1(t),\ldots, f_n(t)\big)$ as
\begin{equation*}
\big(\Tau_i f\big)(t) := f(t) + \Big(\log \int_0^t e^{f_{i+1}(s)-f_i(s)}ds\Big) \big(e_i-e_{i+1}\big)
\end{equation*}
where $e_1,\ldots,e_n$ are basis vectors. Using the $\Tau$ operators, we define operators $\mathcal{S}_1,\ldots, \mathcal{S}_{n-1}$ and the operator $\mathcal{W}$ which act on functions $f$ as
\begin{equation}\label{eq:sdW}
\mathcal{S}_rf:=\Tau_{r}\Tau_{r+1}\cdots \Tau_{n-1}f,\qquad \textrm{and}\qquad
\mathcal{W}f :=\mathcal{S}_{n-1}\mathcal{S}_{n-2}\cdots \mathcal{S}_{1}f.
\end{equation}
\end{defn}

\begin{thm}\label{thm:sdinv}
For any endpoint pair $(U,V)$ and collection of $n$ functions $f=(f_1,\ldots, f_n)$ as in Definition \ref{def:sdpartf},
\begin{equation}\label{eq:sdinv}
f\big[U\to V\big] = \big(\mathcal{W}f\big)\big[U\to V\big].
\end{equation}
\end{thm}

In contrast to discrete case Theorem \ref{thm:dinv}, we do not need to shift the starting points $U$ here.

\begin{proof}
The proof of Theorem \ref{thm:sdinv} follows the same three step program as we used  in proving Theorem \ref{thm:dinv} (and as used to prove \cite[Proposition 4.1]{DOV}). This proof is much closer to that of  \cite[Proposition 4.1]{DOV}. It is only in the first step that there is any real deviation.

\noindent {\bf Step 1 ($n=2$ and $k=1$ case of \eqref{eq:sdinv}).} In this case $f=(f_1,f_2)$, $U=(u,2)$ and $V=(v,1)$. We seek to prove \eqref{eq:sdinv}, which now reads (writing $\Tau$ in place of $\Tau_1$)
\begin{align}
f\big[(u,2)\to (v,1)\big] = \big(\Tau f\big) \big[(u,2)\to (v,1)\big].
\end{align}
From the definition of $\Tau$,
\begin{align}
\big(\Tau f\big)_1(t) = f_1(t) + \log \int_0^t e^{f_2(s)-f_1(s)}ds,\qquad
\big(\Tau f\big)_2(t) = f_2(t) - \log \int_0^t e^{f_2(s)-f_1(s)}ds.
\end{align}
Using this we may rewrite \eqref{eq:sdinv} as the identity
\begin{align}\label{eq:posDOV}
s(u,v)-s(0,v)-s(0,u) = \log\Big(\int_u^v e^{f_2(t)-f_1(t)-2s(0,t)}dt\Big),
\end{align}
where we have defined
\begin{align}
s(x,y) := \log \int_x^y e^{f_2(s)-f_1(s)}ds.
\end{align}
Denoting $g(s)=f_2(s)-f_1(s)$, we can rewrite the integral on the right-hand side of \eqref{eq:posDOV} as
\begin{align}\label{eq:rhsman}
\int_u^v e^{f_2(t)-f_1(t)-2s(0,t)}dt = \int_u^v \frac{e^{g(t)}dt}{\Big(\int_0^t e^{g(s)}ds\Big)^2} =\frac{-1}{\int_0^t e^{g(s)}ds}\Big\vert_u^v = \frac{1}{\int_0^u e^{g(s)}ds}-\frac{1}{\int_0^v e^{g(s)}ds}.\quad
\end{align}
The first equality is simply a rewriting via $g$, the second and third equalities evaluate the integration. It remains to match the result of the above calculation with the exponentiated left-hand side of \eqref{eq:posDOV} which is evaluated as
\begin{align}
\exp\Big(s(u,v)-s(0,v)-s(0,u)\Big) = \frac{\int_{u}^v e^{g(s)}ds}{\int_{0}^v e^{g(s)}ds \int_{0}^u e^{g(s)}ds}.
\end{align}
Since $u<v$, the right-hand side matches the right-hand side of \eqref{eq:rhsman} completing Step 1.
\smallskip

\noindent {\bf Step 2 ($n=2$ and $k\geq 1$ case of \eqref{eq:sdinv}).} Now $f=(f_1,f_2)$, $U=\big\{(u_i,2)\big\}_{i\in \llbracket 1,k\rrbracket}$ and $V=\big\{(v_i,1)\big\}_{i\in \llbracket 1,k\rrbracket}$. This step is shown along the same lines as  \cite[Lemma 4.3]{DOV}. In order that $U$ and $V$ constitute an endpoint pair, it is necessary that the $u$'s and $v$'s are ordered so that connecting them with pairs of non-intersecting paths is possible. The implies that $v_i\leq u_{i+2}$ for $i\in \llbracket 1,k-2\rrbracket$. Define $(z_1,\ldots, z_{2k})$ to be the ordering of the union of $\{u_1,\ldots, u_k\}$ and $\{v_1,\ldots, v_k\}$. We say that the interval $(z_i,z_{i+1}]$ is of ``type 2'' if $z_i=u_{j+1}$ and $z_{i+1}=v_{j}$ for some $j\in \llbracket 1,k-1\rrbracket$.
We say that the interval $(z_i,z_{i+1}]$ is of ``type 1'' if either of the following holds (we temporarily employ the notational convention that $v_0=0$ and $u_{k+1}=\infty$): $z_i=u_j$, $z_{i+1}=u_{j+1}$ and $v_{j-1}<u_j<u_{j+1}<v_{j}$ for some $j\in \llbracket 1,k-1\rrbracket$; or $z_i=u_j$, $z_{i+1}=v_{j}$ and $v_{j-1}<u_j<v_{j}<u_{j+1}$ for some $j\in \llbracket 1,k\rrbracket$.
Finally, we say that the interval $(z_i,z_{i+1}]$ is of ``type 0'' if $z_{i}=v_j$ and $z_{i+1}=u_{j+1}$ for some $j\in \llbracket 1,k-1\rrbracket$.
For $r\in\{0,1,2\}$, if the interval $(z_i,z_{i+1}]$ is of type $r$ this means that for all paths $\pi\in U\to V$, $\pi$ necessarily has exactly $r$ paths traversing that interval. See Figure \ref{fig:Lemma43} for an example.
\begin{figure}[t]
  \captionsetup{width=.8\linewidth}
  \includegraphics[width=3in]{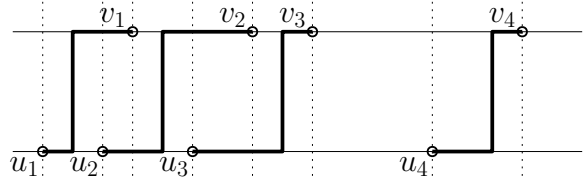}
  \caption{The interval $(v_3,x_4]$ is type 0;  $(u_1,u_2]$, $(v_1,u_3]$, $(v_2,v_3]$ and $(u_4,v_4]$ are type 1; and  $(u_2,v_1]$ and $(u_3,v_2]$ are type $2$. A multipath $\pi\in U\to V$ is shown and the number of paths matches the type of each interval.}
  \label{fig:Lemma43}
\end{figure}

The value of the above partitioning of $(u_1,\ldots, v_n]$ into the $z$ intervals is that on the boundary of each interval, every $\pi\in U\to V$ must take the same values. Thus, the integral which defines $f\big[U\to Y\big]$ factorizes into a product of integrals over each of the $z$ intervals. Consequently, the produces the following decomposition:
\begin{align}
f\big[U\to Y\big] &= \sum_{i=1}^{2k-1} f\Big[\big((z_i,2),(z_i,1)\big)\to \big((z_{i+1},2),(z_{i+1},1)\big) \Big] \mathbf{1}\big\{(z_i,z_{i+1}]\textrm{ is type 2}\big\}  \\
&+\sum_{i=1}^{2k-1} f\big[(z_i,2)\to (z_{i+1},1)\big] \mathbf{1}\big\{(z_i,z_{i+1}]\textrm{ is type 1}\big\}.
\end{align}
Type 0 intervals do not contribute to this sum.
By Step 1, we may replace $f\big[(z_i,2)\to (z_{i+1},1)\big]$ above by $\big(\mathcal{W} f\big)\big[(z_i,2)\to (z_{i+1},1)\big]$ without changing the value. As follows directly from definitions, for the type 2 terms we may replace $f$ by $\mathcal{W} f$ without changing the value. From these two replacements we conclude that $f\big[U\to Y\big] = \big(\mathcal{W} f\big)\big[U\to Y\big]$, completing Step 2.
\smallskip

\noindent {\bf Step 3 ($n\geq 2$ and $k\geq 1$ case of \eqref{eq:sdinv}).} Now $f=(f_1,\ldots, f_n)$, $U=\big\{(u_i,2)\big\}_{i\in \llbracket 1,k\rrbracket}$ and $V=\big\{(v_i,1)\big\}_{i\in \llbracket 1,k\rrbracket}$.
%
%
We claim that for all $m\in \llbracket 1,n-1\rrbracket$,
\begin{align}\label{eq.Taum}
f\big[U\to V\big]= \big(\Tau_m f\big) \big[U\to V\big].
\end{align}
Observe that since $\mathcal{W}$ is written as a composition of $\Tau_m$ operators, applying \eqref{eq.Taum} repeatedly for various values of $m$ yields the desired result \eqref{eq:sdinv}. Note that it is apparent from \eqref{eq.Taum} that the order in which we apply the $\Tau_m$ does not matter. This is different than in the discrete case, and is related to the braid relations discussed in Section \ref{sec:braid}.

To prove \eqref{eq.Taum}, we utilize the following decomposition:
\begin{align}\label{eq.decomptwice}
f\big[U\to V\big] = \log\int_{W,Z} e^{f[U\to W^+]+ f[W \to Z]+ f[Z^- \to V]}dWdZ.
\end{align}
In the above equation, we have taken $W=(w_i,m+1)_{i\in \llbracket 1,k\rrbracket}$, $W^+=(w_i,m+2)_{i\in \llbracket 1,k\rrbracket}$, $Z=(z_i,m)_{i\in \llbracket 1,k\rrbracket}$, and $Z^-=(z_i,m-1)_{i\in \llbracket 1,k\rrbracket}$, the integration is over all $W$ and $Z$ for which $(U,W^+)$, $(W,Z)$, and $(Z^-,V)$ are all endpoint pairs, and the $dW$ and $dZ$ are Lebesgue measure on the $k$-tuples $(w_1,\ldots, w_k)$ and $(z_1,\ldots, z_k)$. The decomposition is valid as written as long as $m\in \llbracket 2, n-2\rrbracket$. To deal with the boundary cases $m=1$ and $m=n-1$ we need to modify the decomposition slightly.
For $m=1$, we set $Z=V$, and drop the integration in $dZ$ and the term $\big[Z^-\to V\big]$ from \eqref{eq.decomptwice}.
For $m=n-1$, we set $W=U$, and drop the integration in $dW$ and the term $\big[U\to W^+\big]$ from \eqref{eq.decomptwice}.
This decomposition is illustrated in Figure \ref{fig:Decomp}.

\begin{figure}[t]
  \captionsetup{width=.8\linewidth}
  \includegraphics[width=3in]{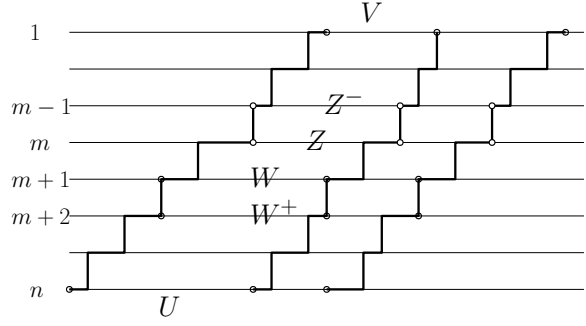}
  \caption{The decomposition (\ref{eq.decomptwice}). The circles on level $1$ correspond to points in $V$, on level $m-1$ to $Z^-$, on level $m$ to $Z$, on level $n+1$ to $W$ and on level $m+2$ to $W^+$. Notice that the set $Z^-$ sits immediately above $Z$ and the set $W^+$ sits immediately below $W$. The boundary cases, when $m=1$ or $m=n+1$ are not shown here.}
  \label{fig:Decomp}
\end{figure}

Now observe that inside the exponential in the integrand on the right-hand side of \eqref{eq.decomptwice}, the term $f[U\to W^+]$ depends only on $f_n,\ldots, f_{m+2}$, the term $f[W \to Z]$ depends only on $f_{m+1}$ and $f_m$ and the term $f[Z^- \to V]$ depends only on $f_{m-1},\ldots,f_1$. Since $\Tau_m$ acts as the identity on all functions in $f$ except $f_m$ and $f_{m+1}$, it follows immediately that $f[U\to W^+]=\big(\Tau_m f\big)[U\to W^+]$ and $f[Z^- \to V] =\big(\Tau_m f\big)[Z^- \to V]$. Meanwhile, using the result of Step 2 it follows that $f\big[W\to Z\big]= \big(\Tau_m f\big)\big[W\to Z\big]$. Putting this together we see that we can replace $f$ by $\big(\Tau_m f\big)$ in the right-hand side of \eqref{eq.decomptwice} without changing the value. Thus, using that decomposition backwards, now with $\big(\Tau_m f\big)$ in place of $f$, we conclude \eqref{eq.Taum}. This completes Step 3 and thus also completes the proof of the theorem.
\end{proof}

\subsection{Relation to geometric RSK correspondence}\label{sec:sdRSK}

As explained at the bottom of page 445 of \cite{ON12} (see also \cite{biane2005}), the operator $\mathcal{W}$ defined in \eqref{eq:sdW} is related to a semi-discrete geometric RSK correspondence. In particular, $\mathcal{W}$ admits the following path formula:
For all $k\in \llbracket 1,n\rrbracket$ and $t\in [0,\infty)$
$$
\sum_{i=1}^{k} \big(\mathcal{W}f\big)_i(t) = f\big[(0,n)^k\to (t,1)^k\big].
$$
\cite[Theorem 3.1 and Corollary 4.1]{ON12} shows that $f_i$ are taken to be independent Brownian motions, then $\big(\mathcal{W}f\big)(t)$  evolves in $t$ as diffusion with generator given by the Doob $h$-transform of the quantum Toda Hamiltonian with $h$ given by the class-one Whittaker function. Based on this result, \cite[Proposition 3.4]{CorHam16} showed that $\big(\mathcal{W}f\big)(t)$ enjoys the $H$-Brownian Gibbs property for an exponential interaction Hamiltonian $H(x)=e^x$.

\subsection{Zero-temperature limit}\label{sec:sdzero}
For $f=(f_1,\ldots, f_n)$ fixed define (note that the $\beta^{-1}$ superscript below is not an exponent, but rather a label for the variant of $f\big[U\to V\big]$ scaled as below)
\begin{equation*}
f\big[U\to V\big]^{\beta^{-1}}:=\beta^{-1} \big(\beta f)\big[U\to V\big].
\end{equation*}
Then, observe that by Laplace's method we can extract a {\it zero-temperature limit}
\begin{equation*}
\lim_{\beta\to \infty} f\big[U\to V\big]^{\beta^{-1}} = \sup_{\pi\in U\to V} \int df\circ \pi=:f\big[U\to V\big]^0.
\end{equation*}
We can also define zero-temperature limits of the operators $\Tau_m$ and $\mathcal{W}$ from Definition \ref{def:sdpit}.
Define $n-1$ operators $\Tau^0_1,\ldots, \Tau^0_n$ which act on $n$-tuples of functions $f_1,\ldots,f_n$ (which we write as $f(t) = \big(f_1(t),\ldots, f_n(t)\big)$ as
\begin{equation*}
\big(\Tau^0_i f\big)(t) := f(t) + \sup_{s\in [0,t]}\big(f_{i+1}(s)-f_i(s)\big)\big(e_i-e_{i+1}\big)
\end{equation*}
where $e_1,\ldots,e_n$ are basis vectors. Using the $\Tau^0$ operators, we define operators $\mathcal{S}^0_1,\ldots, \mathcal{S}^0_{n-1}$ and the operator $\mathcal{W}^0$ which act on functions $f$ as
\begin{equation}\label{eq:sdW}
\mathcal{S}^0_rf:=\Tau^0_{r}\Tau^0_{r+1}\cdots \Tau^0_{n-1}f,\qquad \textrm{and}\qquad
\mathcal{W}^0f :=\mathcal{S}^0_{n-1}\mathcal{S}^0_{n-2}\cdots \mathcal{S}^0_{1}f.
\end{equation}

The following is an immediate corollary of Theorem \ref{thm:sdinv} under the same zero-temperature limit. It was proved earlier as \cite[Proposition 4.1]{DOV}.

\begin{cor}\label{cor:sdzero}
For any endpoint pair $(U,V)$ and collection of $n$ functions $f=(f_1,\ldots, f_n)$ as in Definition \ref{def:sdpartf},
\begin{equation}\label{eq:sdinv}
f^0\big[U\to V\big] = \big(\mathcal{W}^0 f\big)\big[U\to V\big]^0.
\end{equation}
\end{cor}

\section{Some remarks}\label{sec:ques}

\subsection{Braid relations}\label{sec:braid}
The semi-discrete Pitman transform operators $\Tau_1,\ldots, \Tau_{n-1}$ in Definition \ref{def:sdpit} satisfy braid relations \cite{ON12,biane2005} $\Tau_i\Tau_{i+1}\Tau_i = \Tau_{i+1}\Tau_{i}\Tau_{i+1}$. This does not seem to hold for the discrete Pitman transform operators from Definition \ref{defn:dpit}. For instance, when $n=3$, the operator $\mathcal{W} = \Tau_{2,2}\Tau_{2,1}\Tau_{1,2}$ but it is easy to see that it does not equal $\Tau_{1,2}\Tau_{2,1} \Tau_{2,2}$. In fact, the latter composition of operators is not even well-defined due to the changing of the domains under the application of the $\Tau$ operators. The question here is whether there is any sense in which the braid relations generalize to the discrete setting.
\subsection{Airy / KPZ sheet limits}\label{sec:airykpzsheet}

\cite[Proposition 4.1]{DOV} (Corollary \ref{cor:sdzero} herein) serves as the starting point in \cite{DOV} for the construction of the Airy sheet. When $k=1$ and the endpoint pair $(U,V)$ has $U=(u,n)$ and $V=(v,1)$, the result says that the point-to-point last passage time (i.e., $f^0\big[U\to V\big]$\big) for all choices of $u\leq v\mathbb{Z}_{\geq 1}$ can be recovered from the Pitman transform $\mathcal{W}^0 f$ through another last passage time calculation. \cite{DOV} calls $\mathcal{W}^0 f$ the {\it melon} of $f$. When the $f_i$ are Brownian motions, the melon $\mathcal{W}^0 f$ is an $n$-particle Dyson Brownian motion (DBM). This fact was shown in \cite{Bougerol2002, o'connell2002}, generalizing results of \cite{Bary,Gravner2001} which matched the one-time marginal of $\big(\mathcal{W}^0 f\big)_{1}$ to the top eigenvalue of an $n\times n$ GUE random matrix. It can also be seen as a limit of the results mentioned briefly, at the discrete geometric level, in Section \ref{sec:inv}. The Brownian choice of the $f_i$ is known as a {\it solvable} or {\it integrable} model due to its special structural properties such as just mentioned.

Much of  \cite{DOV} is  devoted to showing that under KPZ $n\to \infty$ scaling, the point-to-point last passage time fluctuation has a limit as a two parameter process defined by varying the starting and ending points (the limits of the $u$ and $v$). \cite[Theorem 1.3]{DOV}  claims that the limit, termed the {\it Airy sheet} from \cite{Corwin2015}, can be defined directly (not just as a limit) through the {\it Airy line ensemble} \cite{Prahofer2002,JohPNG,CH14} which, itself, arises as the $n\to \infty$ edge limit of the DBM $\mathcal{W}^0 f$.

\cite{DOV} relies on fine control over particle locations and gaps in DBM as $n\to \infty$ as well as its convergence to the Airy line ensemble. This control, some of which is in the companion paper \cite{DV}, comes from two main tools -- the {\it Brownian Gibbs property} \cite{CH14} and the {\it determinantal} structure \cite{AM} of the DBM and Airy line ensemble.

The question here is whether a similar analysis can be performed based on Theorem \ref{thm:dinv}. It seems most like that in the discrete zero-temperature setting. As described at the end of Section \ref{sec:dzero}, when the weights in last passage percolation are geometric (i.e., the solvable choice of weights), Dyson Brownian motion is replaced by another Markov process which has a discrete version of the Brownian Gibbs property and which is determinantal.

In the positive-temperature cases under the solvable choices of inverse-gamma distributed weights (in the discrete case, as in Section \ref{sec:inv}) or under Brownian $f_i$ (in the semi-discrete case, as in Section \ref{sec:sdRSK}), the $\mathcal{W}D$  and $\mathcal{W}f$ functions still enjoy relatively simple Gibbs properties (see Sections \ref{sec:inv} and \ref{sec:sdRSK} for references to precise statements). However, $\mathcal{W}D$  and $\mathcal{W}f$ are not known to be determinantal point processes. Instead, they relate to {\it Whittaker measures} \cite{ON12,COSZ,Borodin2014}. While there has been much success (for example \cite{Borodin2014,BCF12,BCR,krishnan2018}) in extracting asymptotics for the top-labeled particle of Whittaker measures, there are essentially no results providing control over the lower particles or particle gaps. This poses an immediate impediment to adapting the approach used in \cite{DV,DOV}.

It is known that under special {\it weak noise} scaling, the discrete \cite{alberts2014} and semi-discrete \cite{Nica} positive temperature directed polymer models have limits to the KPZ stochastic PDE. The Gibbsian line ensembles which arise from the solvable models likewise have scaling limits under the special weak noise scaling to the KPZ line ensemble \cite{CorHam16} (see also \cite{O’Connell2016,corwin2017,Wu}). Thus, it is natural to speculate that the KPZ sheet (that we define in a moment) can be recovered directly from the KPZ line ensemble and that this relationship is facilitated through taking a limit of Theorems \ref{thm:dinv} and \ref{thm:sdinv}. The KPZ sheet is easy to define as a function of space-time white noise $\xi$. For $x\in \mathbb{R}$ let $h(t;x,y)=\log Z(t;x,y)$ be the Hopf-Cole solution to the KPZ equation, where $Z$ solves the multiplicative stochastic heat equation
$$
\partial_t Z = \frac{1}{2}\partial_y^2 Z +\xi Z,\qquad h(0;x,y)=\delta_{x=y}.
$$
The KPZ sheet (for a fixed time $t$) is the two-parameter function $(x,y)\mapsto h(t;x,y)$.

\subsection{Other invariances}
Recently there have been two papers \cite{BGW, Petrovinv} which have studied other types of invariances of polymer models and stochastic vertex models. At face-value, those results do not seem directly related to those studied here, though it is enticing to search for a way to unify them. For instance, stochastic vertex models \cite{Corwin2016,Borodin2018} are known to generalize the solvable polymer models. It would be very interesting to find a lifting of our results into that setting. Indeed, there has been plenty of study in \cite{Pei,BorPetnear,MatPet,BufMat} of generalizations of the gRSK correspondence up the hierarchy of Macdonald processes to the $q$-Whittaker or Hall-Littlewood level (the gRSK relates to the Whittaker level) which relate to stochastic vertex models. However, at that level the correspondence is no-longer bijective and instead involves some randomization. This presents an immediate impediment to even formulating an analog of our main results and deserves further consideration.

\bibliographystyle{alpha}
\bibliography{Reference}

\end{document}